\documentclass{article}

\usepackage[utf8]{inputenc}
\usepackage[T1]{fontenc}
\usepackage{geometry}
\usepackage[english]{babel}

\usepackage{amsmath,amssymb,amsthm}
\usepackage{mathrsfs}
\usepackage{dsfont}
\usepackage{color}
\usepackage[colorlinks]{hyperref}
\usepackage{graphicx}
\usepackage{ulem}
\usepackage[dvipsnames]{xcolor}
\usepackage{cancel}
\usepackage{enumitem}
\usepackage{stmaryrd}
\usepackage{nicefrac}
\usepackage{comment}

\usepackage{caption}
\usepackage{subcaption}
\usepackage{cleveref}

\usepackage{chngcntr}

\newtheorem{thm}{Theorem}[section]
\newtheorem{lem}[thm]{Lemma}
\newtheorem{cor}[thm]{Corollary}
\newtheorem{prop}[thm]{Proposition}
\newtheorem{definition}[thm]{Definition}

\newtheorem{rem}[thm]{Remark}
\newtheorem{example}[thm]{Example}

\newcommand{\N}{\mathbb{N}}

\newcommand{\R}{\mathbb{R}}

\newcommand{\bB}{\mathbf{B}}

\newcommand{\bK}{\mathbf{K}}
\newcommand{\bS}{\mathbf{S}}

\newcommand{\bU}{\mathbf{U}}
\newcommand{\bX}{\mathbf{X}}

\newcommand{\vx}{\mathbf{x}}
\newcommand{\vy}{\mathbf{y}}

\newcommand{\vp}{\mathbf{p}}

\newcommand{\vv}{\mathbf{v}}
\newcommand{\vw}{\mathbf{w}}
\newcommand{\vz}{\mathbf{z}}
\newcommand{\vO}{\mathbf{0}}

\newcommand{\balpha}{\boldsymbol\alpha}
\newcommand{\bbeta}{\boldsymbol\beta}

\newcommand{\cO}{\mathcal{O}}
\newcommand{\cK}{\mathcal{K}}
\newcommand{\cL}{\mathcal{L}}
\newcommand{\cM}{\mathcal{M}}
\newcommand{\cP}{\mathcal{P}}
\newcommand{\cQ}{\mathcal{Q}}

\newcommand{\od}{\mathrm{d}}

\renewcommand{\epsilon}{\varepsilon}

\DeclareMathOperator{\argmin}{\mathrm{argmin}}

\newcommand{\SOS}{\mathrm{\Sigma}}
\newcommand{\one}{\mathbf{1}}
\renewcommand{\c}{\mathfrak{c}}
\newcommand{\rank}{\mathrm{rank}}

\newcommand{\norm}[1]{\left\lVert #1\right\rVert}
\newcommand{\norms}[2]{{\norm{#1}}_{#2}}
\newcommand{\abs}[1]{\left|#1\right|}
\newcommand{\dist}{\mathrm{dist}}
\newcommand{\supp}{\mathrm{supp}}

\title{Convergence rate for linear minimizer-estimators in the moment-sum-of-squares hierarchy}
\author{Corbinian Schlosser$^1$}
\date{\today}
\begin{document}

\maketitle
\footnotetext[1]{INRIA - Ecole Normale Supérieure - PSL Research University, Paris, France.}

\begin{abstract}
    Effective Positivstellensätze provide convergence rates for the moment-sum-of-squares (SoS) hierarchy for polynomial optimization (POP). In this paper, we add a qualitative property to the recent advances in those effective Positivstellensätze. We consider optimal solutions to the moment relaxations in the moment-SoS hierarchy and investigate the measures they converge to. It has been established that those limit measures are the probability measures on the set of optimal points of the underlying POP. We complement this result by showing that these measures are approached with a convergence rate that transfers from the (recent) effective Positivstellensätze. As a special case, this covers estimating the minimizer of the underlying POP via linear pseudo-moments. Finally, we analyze the same situation for another SoS hierarchy -- the upper bound hierarchy -- and show how convexity can be leveraged.
\end{abstract}

\section{Introduction}

Since its introduction around the year 2000 in the works~\cite{lasserre2001global,parrilo2000structured}, the moment-sum-of-squares (SoS) hierarchy has demonstrated remarkable success and found applications in numerous fields. Originally developed for addressing polynomial optimization problems (POPs) of the form
\begin{equation}\label{intro:pop}
    f^\star := \min\limits_{\vx \in \bK} f(\vx),
\end{equation}
where \( f \) is a polynomial and \( \bK \) is a semialgebraic set -- this hierarchy has since been extended to a wide range of applications. These include stability analysis and control of dynamical systems~\cite{parrilo2000structured,korda2014convex,lasserre2008nonlinear}, graph theory~\cite{de2002approximation}, game theory~\cite{laraki2012semidefinite}, and quantum information theory~\cite{fang2021quantum,klep2023state}, among others. For a comprehensive overview, see~\cite{lasserre2015introduction,lasserre2009moments,henrion2020moment}.

The moment-SoS hierarchy comprises two dual sequences of finite-dimensional semidefinite programs: the \textit{moment relaxations}, whose decision variables correspond to pseudo-moments, and the \textit{SoS tightenings}, which employ sums-of-squares polynomials to certify positivity. A key feature of this hierarchy, when applied to solving (\ref{intro:pop}), is its construction of a sequence of finite-dimensional convex optimization problems whose optimal values converge monotonically to the global minimum \(f^\star\). However, while the convergence of these objective values is guaranteed under mild assumptions, extracting optimal points \(\vx^\star \in \bK\) from the moment-SoS hierarchy poses significant challenges. Positive results are available for (generic!, see \cite{nie2013certifying,baldi2024exact}) cases where the \textit{flatness condition} (see Definition~\ref{def:flatness}) holds, enabling exact extraction of minimizers~\cite{curto2000truncated}. Nevertheless, determining in advance whether flatness will occur for a given problem instance is generally computationally hard~\cite{vargas2024hardness}. Thus, the optimal point extraction from the hierarchy invites further investigations.

This paper builds upon and extends existing literature in this domain, focusing on the moment side of the hierarchy. Specifically, it refines the results of~\cite[Theorem 12]{schweighofer2005optimization}, where it is shown that the solutions of moment relaxations converge to the moments of measures supported on the set of optimal points for (\ref{intro:pop}). Our contribution enhances this analysis by providing a quantitative convergence rate, leveraging insights from~\cite{baldi2021moment}. For cases where (\ref{intro:pop}) has a unique minimizer \(\vx^\star \in \bK\), we establish that this minimizer can be approximated, with a polynomial convergence rate, using linear pseudo-moments obtained from the solutions within the moment hierarchy. Furthermore, we show that certain convexity properties of the POP (\ref{intro:pop}) translate into qualitative guarantees for the approximation of $\vx^\star$. These include feasibility of the approximation and a priori bounds on its associated cost (see Section \ref{sec:UpperboundHierarchy}).

Additionally, we show that certain convexity properties of the POP (\ref{intro:pop}), transfer to certain qualitative results on the approximations of the minimizer $\vx^\star$, including feasibility and apriori bounds on its cost (see Section \ref{sec:UpperboundHierarchy}). This part is closely linked to \cite[Section 13]{lasserre2015introduction}.

A central element of our analysis involves the use of Positivstellensätze, such as Schmüdgen’s and Putinar’s theorems~\cite{putinar1993positive,schmudgen2017moment}, which provide structural decompositions of positive polynomials as sums of squares. Quantitative refinements of these results, as developed in~\cite{nie2007complexity}, have enabled bounds on convergence rates within the moment-SoS hierarchy. Building on these foundations, the recent work~\cite{baldi2021moment} introduced tighter bounds, establishing polynomial growth rates for effective versions of Putinar’s Positivstellensatz and exploring the dual perspective on pseudo-moment sequences.

Compared to~\cite{baldi2021moment}, our analysis focuses on approximating pseudo-moments of the moment-SoS hierarchy by measures supported specifically on the set of optimal points for the POP (\ref{intro:pop}). This offers an enhanced characterization of the convergence behavior compared to~\cite[Section 5]{baldi2021moment} where the limit measures are instead supported on the larger semialgebraic set $K$ associated with the POP (\ref{intro:pop}).

The remainder of this paper is structured as follows. Section~\ref{sec:Preliminaries} reviews the moment-SoS hierarchy and Putinar’s Positivstellensatz, laying the groundwork for subsequent developments. Section~\ref{sec:ConvRatesHierarchy} recalls key quantitative results on convergence rates from the literature. In Section~\ref{sec:MinExtraction}, we summarize existing techniques for extracting minimizers from solutions within the moment hierarchy. Our main contributions are detailed in Section~\ref{sec:MomentConvergence}, where we analyze the quantitative convergence of moment relaxations and establish polynomial convergence rates for certain moment-based estimators of the minimizer of (\ref{intro:pop}). Section~\ref{sec:UpperboundHierarchy} extends this analysis to an upper-bound hierarchy and highlights its interplay with convexity properties of the original POP. Finally, Section~\ref{sec:alternatives} explores alternative approaches to approximating the minimizers of (\ref{intro:pop}), discussing limitations and future directions, before concluding in Section~\ref{sec:Conclusion}.

\section{Notation}

We denote the reals by $\R$ and the natural numbers by $\N$. We make the convention of writing vectors or sets of vectors in bold font and scalars in regular font. For $n\in \N$, we denote by $\R[X_1,\ldots,X_n]$, or short $\R[\bX]$, the space of polynomials in $n$ variables with coefficients in $\R$. For a polynomial $f\in \R[\bX]$, we denote by $\deg (f)$ the degree of $f$. The space of polynomials of degree at most $d\in \N$ is denoted by $\R[\bX]_d$. For a multi-index $\balpha = (\alpha_1,\ldots,\alpha_n) \in \N^n$, $|\balpha| := \alpha_1+\ldots+\alpha_n$ is the range of $\balpha$ and $\bX^{\balpha} := X_1^{\alpha_1}\cdots X_n^{\alpha_n}$ is the corresponding monomial. For a measurable set $\bK \subset \R^n$ the set $\cM(\bK)$ denotes the set of Borel measures on $\bK$. The Lebesgue measure is always denoted by $\lambda$ and its restriction to a set $\bK$ by $\lambda\big|_{\bK}$. For two measures $\mu,\nu \in \cM(\bK)$ we write $\od \nu = g\ \od \mu$ if $\nu$ has density $g$ with respect to $\mu$.

\section{Preliminaries}\label{sec:Preliminaries}

\subsection{Moment-SoS hierarchy for polynomial optimization}

We consider a constraint optimization problem  of the form
\begin{eqnarray}\label{eq:MinfProblem}
    f^* := \inf\limits_{\vx} & f(\vx) &\\
	\text{s.t.} & \vx\in \bK & \notag
\end{eqnarray}
where $f \in \R[\bX]$ is a polynomial and $\bK\subset \R^n$ a given set. To leverage the machinery from real algebraic geometry we assume that the constraint set $\bK$ is closed basic semialgebraic, i.e. it can be described by polynomial inequalities.

\begin{definition}\label{def:BasicSemialgebraic}
    A subset $\bK \subset \R^n$ is called closed basic semialgebraic if there exists $m \in \N$ and polynomials $p_1,\ldots,p_m \in \R[\bX]$ such that $\bK$ has the representation
    \begin{equation}\label{EquKSemialgebraic}
	    \bK = \cK(p_1,\ldots,p_m):= \{\vx \in \R^n : p_1(\vx)\geq 0,\ldots,p_m(\vx)\geq 0\}.
    \end{equation}
    The set $\bK$ is called closed semialgebraic if it is a finite union of closed basic semialgebraic sets. We will write $\bK = \cK(\vp)$ to shorten the notation.
\end{definition}

For the minimization problem (\ref{eq:MinfProblem}) this leads to the notion of polynomial optimization.

\begin{definition}[Polynomial optimization problem (POP)]\label{prelim:def:POP}
    The problem
    \begin{eqnarray}\label{eq:pop}
	    f^* := \inf\limits_{\vx} & f(\vx) &\\
	\text{s.t.} & \vx \in \bK. & \notag
    \end{eqnarray}
    with a polynomial $f$ and $\bK \subset \R^n$ a closed basic semialgebraic set is called a polynomial optimization problem.
\end{definition}

Polynomial optimization is linked to the analysis of positive polynomials via the following simple reformulation of the problem (\ref{eq:pop})
\begin{eqnarray}\label{eq:popPos}
    f^* = \sup\limits_{s \in \R} & s &\\
	\text{s.t.} & f - s \geq 0 \text{ on } \bK. \notag
\end{eqnarray}
Solving the formulation (\ref{eq:popPos}) requires testifying whether the polynomial $f-s$ is non-negative on $\bK$. At present, there are no computationally efficient characterizations for the cone of non-negative polynomials on $\bK$. However, for certain sub-cones, membership can be verified efficiently. The moment-SoS hierarchy~\cite{lasserre2001global} builds upon one such sub-cone -- the cone of SoS polynomials denoted by
\begin{equation}\label{prelim:SOSpolynomials}
    \SOS := \left\{ \sum\limits_{i = 1}^m q_i^2 : m \in \N, \; q_1,\ldots,q_m \in \R[\bX] \right\}.
\end{equation}
The polynomials $q\in \SOS$ are globally non-negative, to specify non-negativity to a closed semialgebraic set $\cK (\vp)$ for a vector of polynomials $\vp = (p_1,\ldots,p_m)$ we consider the quadratic module generated by $p_1,\ldots,p_m$.

\begin{definition}\label{def:QuadraticModule}
    For $p_1,\ldots,p_m \in \R[\bX]$, the quadratic module $\cQ(\vp)$ is defined by
    \begin{equation*}
        \cQ(\vp) := \left\{\sigma_0 + \sum\limits_{i = 1}^m \sigma_i p_i : \sigma_0,\sigma_1,\ldots,\sigma_m \in \SOS\right\}.
    \end{equation*}
    For $d\in \N$, we denote by $\cQ_d(\vp) \subset \R[\bX]_d$ the truncated quadratic module given given by
    \begin{eqnarray}\label{eq:quadrmodd}
        \cQ_d(\vp) := \bigg\{\sigma_0 + \sum\limits_{i = 1}^m \sigma_i p_i & : & \sigma_0,\sigma_1,\ldots,\sigma_m \in \SOS,\\
        & &\deg(\sigma_0),\deg(\sigma_1 p_1),\ldots,\deg(\sigma_m p_m) \leq d\bigg\}. \notag
    \end{eqnarray}
\end{definition}

If the closed basic semialgebraic set $\bK$ is given by $\bK = \cK(\vp)$ as in (\ref{EquKSemialgebraic}) then all polynomials in $\cQ(\vp)$ are non-negative on $\bK$. This motivated the following SoS hierarchy~\cite{lasserre2001global}. 

\begin{definition}[SoS hierarchy for polynomial optimization;~\cite{lasserre2001global}]\label{def:lasserrehierarchy}
    Let $f \in \R[\bX]$. The SoS hierarchy for the POP
    \begin{eqnarray*}
	    f^* := \inf\limits_{\vx \in \R^n} & f(\vx) &\\
	\text{s.t.} & \vx \in \cK(\vp) & \notag
\end{eqnarray*}
    is given by the following sequence of sum-of-squares programs: For each $d \in \N$ the $d$-th level of the SoS hierarchy is given by the optimization problem
    \begin{eqnarray}\label{eq:LasserreHierarchy}
	    f_d^* := \sup\limits_{s \in \R} & s &\\
	\text{s.t.} & f-s \in \cQ_d(\vp). & \notag
\end{eqnarray}
\end{definition}

\begin{rem}\label{prelim:RemarkLasserreHierarchy}
    For each $d \in \N$ the optimization problem (\ref{eq:LasserreHierarchy}) can be formulated as an SDP, see~\cite{lasserre2015introduction,parrilo2000structured,laurent2009sums} for surveys on polynomial optimization.
\end{rem}

Clearly, one should ask whether $f_d^*$ converges to $f^*$. This question is answered by the following theorem.

\begin{thm}[{Convergence of the SoS hierarchy;~\cite{lasserre2001global}}]\label{thm:ConvSoSHierarchy}
    Assume that $R - \norms{\bX}{2}^2 \in \cQ(\vp)$ for some $R\geq 0$ (Archimedean property). Then it holds
    \begin{equation*}
        f_d^* \leq f_{d+1}^* \text{ for all } d \in \N \text{ and } f_d^* \text{ converges to the \textbf{global} optimum }  f^* \text{ as } d\rightarrow \infty.
    \end{equation*}
\end{thm}

\paragraph*{Moment approach and duality for the semidefinite programs}

The optimization problem (\ref{eq:popPos}) is a \textit{linear} optimization problem in $s$ (with an ``infinite dimensional constraint'' $f(\vx) - s \geq 0$ on $\bK$). Therefore it is subject to duality. We refer to the dual problem of (\ref{eq:popPos}) by \textit{the measure formulation} of the POP (\ref{eq:pop}), it reads
\begin{eqnarray}\label{eq:DualpopPos}
    m^\star = \inf\limits_{\mu} & \int\limits_K f \; d\mu &\\
	\text{s.t.} & \mu \in \cM(\bK)_+ & \notag\\
			    & \mu(\bK) = 1 & \notag
\end{eqnarray}
where $\cM(\bK)_+$ denotes the set of non-negative measures on $\bK$. Furthermore, it holds $m^\star = f^\star$~\cite{lasserre2015introduction}. The duality between (\ref{eq:popPos}) and (\ref{eq:DualpopPos}) is based on the duality between the cone of non-negative polynomials on $\bK$ and the cone of non-negative measures on $\bK$ where each measure $\mu \in \cM(\bK)$ is identified with a (positive) linear form
\begin{equation}\label{eq:DefLMu}
    L_\mu:\R[\bX] \rightarrow \R, \quad L_\mu(p) := \int\limits p \; \od \mu.
\end{equation}
This duality transfers to the SoS hierarchy via the cone $\cQ_d(\vp)$ and its dual cone $\cQ^*_d(\vp)$ given by
\begin{equation}\label{eq:def:QuadModDual}
    \cQ_d(\vp)^* := \{ L:\R[\bX]_{d} \rightarrow \R : L \text{ linear }, L(q) \geq 0 \text{ for all } q \in \cQ_d(\vp)\}.
\end{equation}

\noindent For $d\in \N$, the dual problem to (\ref{eq:LasserreHierarchy}) is given by
\begin{eqnarray}\label{eq:MomentHierarchy}
    m_d^* := & \inf\limits_{L} & L(f)\\
	& \text{s.t.} & L \in \cQ_d(\vp)^*\notag\\
& & L(1) = 1. \notag
\end{eqnarray}

Analog to the SoS hierarchy, we refer to the hierarchy of optimization problems (\ref{eq:MomentHierarchy}) as the moment hierarchy.

\begin{rem}
    For $d\in \N$ the cone $\cQ_d(\vp)^*$ does not necessarily consist of truncations of measures (\ref{eq:DefLMu}) and we identify its elements $L\in \cQ_d(\vp)^*$ with the associated family so-called pseudo-moments $L\left(\bX^{\balpha}\right)$ for $\balpha\in \N^n$ with $\abs{\balpha}\leq d$. In general, $\cQ_d(\vp)^*$ is a strict superset of the cone of truncated measures, i.e. the following inclusion is strict
    \begin{equation*}
        \cQ_d(\vp)^* \supset \{ (L_\mu)\big|_{\R[\bX]_{d}}: \mu \in \cM(\bK) \text{ and } L_\mu \text{ given by (\ref{eq:DefLMu})}\}. 
    \end{equation*}
    For this reason, each level (\ref{eq:MomentHierarchy}) in the moment hierarchy is also called a moment-relaxation and it holds $m_d^* \leq f^\star$. Dually, the SoS hierarchy (\ref{eq:LasserreHierarchy}) is also referred to as SoS-tightening.
\end{rem}

By weak duality, it holds $f_d^* \leq m_d^* \leq f^*$, and from Theorem \ref{thm:ConvSoSHierarchy} we get the following corollary.

\begin{cor}\label{cor:ConvMomHierarchy}
    Under the assumption from Theorem \ref{thm:ConvSoSHierarchy} it holds
    \begin{equation*}
        f_d^* \leq m_d^* \nearrow f^* \quad \text{as} d\rightarrow \infty.
    \end{equation*}
\end{cor}

\begin{rem}
    In~\cite{tacchi2022convergence}, the convergence of the moment hierarchy (\ref{eq:MomentHierarchy}) is guaranteed for more general conditions than the Archimedean condition from Theorem \ref{thm:ConvSoSHierarchy}.
\end{rem}

% \begin{rem}[Pseudo-moments]\label{rem:Pseudomoments}
%     Representing an element $L\in \cQ_d(\vp)^*$ via its so-called pseudo-moments $y_{\balpha} := L\left(\bX^{\balpha}\right)$ for $\balpha \in \N^n$ with $|\balpha|\leq 2d$ allows to represent $\cQ_d(\vp)^*$ via linear matrix inequalities. This standard approach can be found in any textbook on polynomial optimization, such as~\cite{lasserre2015introduction,laurent2009sums} among others. The non-negativity conditions $L(q) \geq 0$ for $q\in \cQ_d(\vp)$ transforms into the following conditions on its pseudo-moment vector $(\vy_\alpha)_{\abs{\balpha}\leq d}$
%     \begin{equation*}
%         M_d(\vy) \succeq 0, M_d(y\cdot p_i) \succeq 0 \quad \text{ for } i = 1,\ldots,m,
%     \end{equation*}
%     where the matrices $M_d(\vy), M_d(\vy\cdot p_i)$ are given by
%     \begin{equation*}
%         (M_d(\vy))_{\balpha,\bbeta} := L(\bX^{\balpha} \cdot \bX^{\bbeta}) = y_{\balpha + \bbeta} \quad \text{ and } \quad (M_d(\vy \cdot p_i))_{\bgamma,\veta} :=  L(\bX^{\bgamma} \cdot p_i \cdot \bX^{\veta})
%     \end{equation*}
%     for $\balpha,\bbeta,\bgamma,\veta \in \N^n$ with $|\balpha|, |\bbeta| \leq d$ and $|\bgamma|,|\veta| \leq d-\deg (p_i) $ for $i = 1,\ldots,m$  respectively.
% \end{rem}

\subsection{Putinar's Positivstellensatz}
The convergence results from Theorem \ref{thm:ConvSoSHierarchy} and Corollary \ref{cor:ConvMomHierarchy} are consequences of the celebrated Putinar's Positivstellensatz~\cite{putinar1993positive}. It states that under a certain compactness condition (Archimedean property) on the set $K =\cK(\vp)$ any \textit{strictly positive} polynomial on $\bK$ is contained in the quadratic module $\cQ(\vp)$.

\begin{thm}[Putinar's Positivstellensatz {\cite[Theorem 1.3 \& Lemma 3.2]{putinar1993positive}}] \label{thm:putinar}
    Let $m \in \N$ and $\vp\in \R[\bX]$. If there exists $R \geq 0$ s.t. $R^2 - \norms{X}{2}^2 \in \cQ(\vp)$, then any polynomial $f\in \R[\bX]$ that is strictly positive on $\cK(\vp)$ belongs to $\cQ(\vp)$.
\end{thm}

We end this section with a few short remarks on Putinar's Positivstellensatz.

\begin{rem}[On the Archimedean property] \label{rem:archi} \leavevmode
In practice, in many cases the set $\bK:= \cK (\vp)$ is compact and a radius $R\geq 0$ with $\bK \subset \bB_R(0)$ is known. In that case, the Archimedean property can be enforced by simply adding the constraint $p_{m+1}(\vx) := R^2- \norms{\vx}{2}^2$.
\end{rem}

\begin{rem}[Need for \textit{strict} positivity]
    If, in Putinar's Positivstellensatz, strict positivity is replaced by non-negativity the statement does not remain true. A famous example of a non-negative polynomial that is not sum-of-squares is the Motzkin polynomial~\cite{motzkin1967arithmetic}. Even more is true -- there are many more non-negative polynomials on $\cK(\vp)$ than in the quadratic module $\cQ(\vp)$, see~\cite{blekherman2006there}.
\end{rem}

\begin{rem}[Effective Positivstellensätze]
    Compared to the structural statement, Theorem \ref{thm:putinar}, on positive polynomials, effective versions of Putinar's Positivstellensatz address bounding $d\in \N$ for which a positive polynomial $f$ belongs to the truncated quadratic module $\cQ_d(\vp)$ defined in (\ref{eq:quadrmodd}). We will state a recent effective Positivstellensatz from~\cite{baldi2021moment} in Theorem \ref{thm:OldBaldi}.
\end{rem}

\subsection{Convergence rates for the moment-SoS hierarchy}\label{sec:ConvRatesHierarchy}

In this section, we recall some of the important concepts and recent developments concerning the convergence rate of $f_d^\star \rightarrow f^\star$ respectively $m_d^\star \rightarrow f^\star$ as $d$ tends to infinity. The main result in this paper is that the convergence rates relate to convergence rates for (candidate) minimizers of (\ref{eq:pop}) extracted from the moment-SoS hierarchy. 

\paragraph{Flatness and finite convergence}\label{Sec:Flatness}
The moment hierarchy, respectively the moment-relaxations (\ref{eq:MomentHierarchy}),  enjoys a powerful stopping criterion certifying finite convergence. This is the so-called flatness criterion. To formulate it, we introduce the bilinear form $B^L:\R[\bX]_{d}\times \R[\bX]_d \rightarrow \R$, associated to a linear form $L:\R[\bX]_{2d}\rightarrow \R$, given by
\begin{equation*}
    B^L(g,h) := L(g\cdot h) \quad \text{for } g,h\in \R[\bX]_d.
\end{equation*}
The bilinear form $B^L$ has a representing matrix $M^L = (M^L_{\alpha,\beta})_{\abs{\balpha},|\bbeta| \leq d}$, the so-called moment matrix. % We have encountered the moment matrix $M^L$ already in Remark \ref{rem:Pseudomoments}.
It is given by
\begin{equation}\label{eq:MLBilinearForm}
    M^L_{\alpha,\beta} := B^L(\bX^{\balpha},\bX^{\bbeta}) = L(\bX^{\balpha} \cdot  \bX^{\bbeta}) = L(\bX^{\balpha + \bbeta}) \quad \text{for } \abs{\balpha},\abs{\bbeta} \leq d
\end{equation}
and for $g = \sum\limits_{\abs{\balpha}\leq d} v_{\balpha} \bX^{\balpha},\, h = \sum\limits_{\abs{\balpha}\leq d} w_{\balpha} \bX^{\balpha} \in \R[\bX]_d$ it holds
\begin{equation}\label{eq:MRepMatrixL}
    L(g\cdot h) = \vv^T M^L \vw
\end{equation}
where $\vv = (v_{\balpha})_{\abs{\balpha} \leq d}, \vw = (w_{\balpha})_{\abs{\balpha} \leq d}$.
% Hence it holds for $p,q \in \R[\bX]$ written as $p = \sum\limits_{\alpha \in \N^n} v_\alpha x^{\balpha}$ and $q = \sum\limits_{\alpha \in \N^n} w_\alpha x^{\balpha}$ where only finitely many $p_\alpha, q_\alpha$ are non-zero
% \begin{equation*}
%     L(p\cdot q) = B^L(p,q) = v^T M w
% \end{equation*}
% where $v = (v_\alpha)_{\alpha \in \N^n}, w = (w_\alpha)_{\alpha \in \N^n}$.
%When $L:\R[\bX]_{2d}\rightarrow \R$ has a representing measure $\mu$, that is $L(q) = \int\limits q \; \od \mu$ holds for all $q\in \R[\bX]_{2d}$, then $M^L$ coincides with 
By commutativity of $\R[\bX]_d$, the bilinear form $B^L$ (respectively the moment matrix $M^L$) is symmetric. Further, $B^L$ (respectively $M^L$) is positive semidefinite if and only if $L\in \cQ(1)_{2d}^*$. 

The following flatness criterion concerns the rank of the matrix $M^L$, see for instance~\cite[Definition 2.39]{lasserre2015introduction} or~\cite{curto2000truncated,laurent2009sums}.

\begin{definition}[{r-Flatness}]\label{def:flatness}
    Let $r,d \in \N$ with $r \leq d$. A linear form $L:\R[\bX]_{2d}\rightarrow \R$ is $r$-flat if it holds
    \begin{equation*}
        \rank \ M^{L} = \rank \ M^{L_{-r}}
    \end{equation*}
    where $L_{-r}:= L\big|_{\R[\bX]_{2(d-r)}}:\R[\bX]_{2(d-r)}\rightarrow \R$ denotes the restriction of $L$ to $\R[\bX]_{2(d-r)}$.
\end{definition}

Flatness is a powerful tool to guarantee finite convergence of the moment hierarchy, as the following Theorem emphasizes.

\begin{thm}[{\cite[Theorem 6.18]{laurent2009sums}}]\label{thm:flatness1}
    Set $r := \max\limits_{i} \deg (p_i)$ and let $d\geq \max \{\deg(f),r\}$. If $L^\star_d$ is optimal for (\ref{eq:MomentHierarchy}) and $r$-flat then it holds $m_d^\star = f^\star$, i.e. finite convergence for the moment hierarchy.
\end{thm}

From a practical perspective, flatness is a useful criterion because it can be tested directly for the computed solution $L^\star_d$ of (\ref{eq:MomentHierarchy}). Furthermore, flatness occurs generically for POPs~\cite{nie2013certifying,baldi2024exact}. However, it is computationally hard to decide if or at which level $d$ finite convergence occurs for a given instance of POP~\cite{vargas2024hardness}. Consequently, general asymptotic convergence rates for the convergence of $m_d^\star$ respectively $f^\star_d$ to $f^\star$ as $d\rightarrow \infty$ remain of practical importance. In the next section, we state such convergence rates based on state-of-the-art effective Positivstellensätze from~\cite{baldi2021moment}.

\paragraph{Effective Putinar's Positivstellensatz}\label{sec:EffPosSätze}

In this text, we will use an effective version of Putinar's Positivstellensatz from~\cite{baldi2021moment}. First, we need to introduce the \L{}ojasiewicz exponent.

\begin{thm}[\L{}ojasiewicz inequality {\cite{lojasiewicz1959probleme}\cite[Corollary 2.6.7]{bochnak2013real} }] \label{thm:loja} \leavevmode
    Let $\bK$ be a bounded closed semialgebraic set and $f,g$ be continuous semialgebraic functions (see Definition \ref{def:SemialgebraicFunction}) with $f^{-1}(\{0\}) \subset g^{-1}(\{0\})$. Then there exist $\c,\text{\L} \geq 0$ with
    \begin{equation*}
        |g(\vx)|^\text{\L} \leq \c |f(\vx)| \quad \text{for all }\vx \in \bK.
    \end{equation*}
\end{thm}

Later, we will need to choose the exponent $\text{\L}$ in Theorem \ref{thm:loja} to be at least $1$. This is always possible, as stated in the following lemma.

\begin{lem}\label{lem:LojaGeq1}
    The exponent $\text{\L}$ in Theorem \ref{thm:loja} can be chosen with $\text{\L} \geq 1$.
\end{lem}

\begin{proof}
    Let $\c,\text{\L} \geq 0$ be as in Theorem \ref{thm:loja}. If $\text{\L} \geq 1$ we are done. If $\text{\L} < 1$ then $|g(\vx)|^{\text{\L}} \geq |g(\vx)|$ for all $\vx$ in the open set $\bU:= \{\vx \in \bK: |g(\vx)| < 1\}$, i.e. the \text{\L}ojasiewicz inequality holds on $\bU$ with $\text{\L} = 1$ and the same constant $\c$. To assure it holds also on $\bK\setminus \bU$ for $\text{\L} = 1$ we will modify the constant $\c$. By the assumption on $f$ and $g$, we have $\bU \supset g^{-1}(\{0\}) \supset f^{-1}(\{0\})$. Hence $f(\vx) \neq 0$ for all $\vx \in \bK \setminus \bU$. By compactness of $\bK\setminus \bU$ and continuity of $f$ and $g$ there exists $\varepsilon > 0$ and $M\geq 0$ such that $|f(\vx)|\geq \varepsilon$ and $\abs{g(\vx)}\leq M$ for all $\vx \in \bK\setminus \bU$. For all $\vx\in \bK \setminus \bU$ we get
    \begin{eqnarray*}
        |g(\vx)| & \leq & M = \frac{M}{\varepsilon} \varepsilon \leq \Tilde{\c} |f(\vx)|
    \end{eqnarray*}
    for $\Tilde{\c}:= \frac{M}{\varepsilon}$. Setting $\c' := \max\{\c,\Tilde{\c}\}$ shows that $|g(\vx)| \leq \c' |f(\vx)|$ for all $\vx\in \bK$, i.e. the \text{\L}ojasiewicz inequality with $\text{\L} = 1$.
\end{proof}

For the statement of the effective Putinar's Positivstellensatz in Theorem~\ref{thm:OldBaldi} we use the following notation from~\cite{baldi2021moment}
\begin{equation}\label{eq:NormFHypercube}
    \|f\| := \max\limits_{\vx \in [-1,1]^n} |f(\vx)|.
\end{equation}

\begin{thm}[Effective Putinar Positivstellensatz {\cite[{\bf Theorem 1.7}]{baldi2021moment}}]\label{thm:OldBaldi} \leavevmode

Let $n \geq 2$, $f,p_1,\ldots,p_m \in \R[\bX]$ with $f> 0$ on $\cK(\vp)$. Assume that
\begin{equation}\label{eq:ArchimAndNormalisationLorenzo}
    1 - \|X\|^2 \in \cQ(\vp) \quad \text{and} \quad \forall 1\leq i \leq m, \quad \|p_{i}\| := \max\limits_{\vx \in [-1,1]^{n}} |p_{i}(\vx)| \leq \frac{1}{2}.
\end{equation}
Then one has
\begin{equation}\label{eq:effputinarOld}
d \geq \gamma(n,\vp) \deg(f)^{3.5m\text{\L}} \left(\nicefrac{\|f\|}{f^\star}\right)^{2.5m\text{\L}} \quad \Longrightarrow \quad f \in \cQ_d(\vp)
\end{equation}
for $1 \leq \gamma(n,\vp) \leq \Gamma \, n^3 \, 2^{5\text{\L}-1} \,m^n \, \c^{2n} \, \max\limits_{i}\deg(p_i)^n$ with the constants $\c, \text{\L}$  depending only on $f,p_1,\ldots,p_m$, and the constant $\Gamma > 0$ does not depend on $n,f,p_1,\ldots,p_m$.
\end{thm}

The effective Putinar's Positivstellensatz Theorem \ref{thm:OldBaldi} immediately translates into a convergence rate for the SoS-tightenings $f^\star_d$ from (\ref{eq:LasserreHierarchy}).

\begin{cor}[Convergence rate for POP {\cite[{\bf Theorems 4.2 \& 4.3}]{baldi2021moment}}] \label{cor:poprate} \leavevmode
    Let $n \geq 2$. Under the assumptions on $p_1,\ldots,p_m$ from Theorem \ref{thm:OldBaldi}, one has
    \begin{equation} \label{eq:poprate}
        0 \leq  f^\star - m^\star_d \leq f^\star - f^\star_d \in \cO\left( d^{- 1 / 2.5 n \text{\L}}\right)
    \end{equation}
    where $\text{\L}$ is the constant from Theorem \ref{thm:OldBaldi}.
\end{cor}

In this section, we were concerned with the convergence of $m^\star_d$ from (\ref{eq:MomentHierarchy}) and $f^\star_d$ from (\ref{eq:LasserreHierarchy}) to the optimal value $f^\star$ of (\ref{eq:pop}). The next section is devoted to the extraction and approximation of optimal points of (\ref{eq:pop}) based on optimal points of the moment-relaxation (\ref{eq:MomentHierarchy}).

\section{Minimizer extraction}\label{sec:MinExtraction}

In this section, we treat two established ways of extracting (approximate) minimizers for (\ref{eq:pop}) from solutions $L^\star_d$ of the moment-relaxation (\ref{eq:MomentHierarchy}). One is concerned with finding exact minimizers, i.e. optimal points $\vx^\star \in \bK$ for (\ref{eq:pop}) with $f(\vx^\star) = f^*$, and the other proposes candidate points $\vx^{(d)}$ which hopefully lie close to optimal points $\vx^\star$. Clearly, exactly extracting minimizers $\vx^\star$ is desirable. However, we should not expect that this is always tractable, as many computationally complex problems can be modeled by polynomial optimization problems, see for instance~\cite{laurent2005semidefinite}. Thus, for exactly extracting optimal solutions $\vx^\star$, one restricts to the (generic!) case of so-called flatness, where finite convergence of the moment hierarchy is obtained \cite{nie2014exact,baldi2024exact}. In those situations, optimal solutions $L^\star_d$ for (\ref{eq:MomentHierarchy}) arise from atomic measures supported on optimal points, and the atoms can be computed efficiently. For the case, when the convergence is not finite (or at least when we do not know if finite convergence occurs) the desirable exact extraction of minimizers does not apply. For this reason, we investigate also a second method complementing the exact extraction method. It provides an easy-to-use method based only on linear moments of optimal solutions $L^\star_d$ of the moment-relaxation (\ref{eq:MomentHierarchy}) and induces candidate minimizers for (\ref{eq:pop}). The central topic in this text concerns the quality of those candidate minimizers with a particular focus on asymptotic convergence rates towards global minimizers of the POP (\ref{eq:pop}).

 \begin{rem}\label{rem:ComplexityFeasiblePoint}
    Already the task of finding a feasible point $\vx\in \bK$ can be highly complex. An example from theoretical computer science is the NP-hard satisfiability problem SAT, see~\cite[Theorem 2.10]{arora2009computational}, which addresses the existence of a point $\vx$ in a semialgebraic set $\bK$ (more precisely, a subset of the boolean hypercube). Consequently, extracting exact minimizers to POP (\ref{eq:pop}) is computationally challenging.
\end{rem}

\subsection{Extracting exact minimizers}\label{Sec:ExtractExactMinimizers}

In this section, we briefly outline an established method for extracting minimizers from flat solutions of the moment-relaxation \ref{eq:MomentHierarchy}. In the Appendix \ref{Appendix:Flatness}, we accompany this introduction with some illustrative examples. For details on the method, we refer to \cite{henrion2005detecting,lasserre2008semidefinite} and survey texts~\cite[Section 6.1.2]{lasserre2015introduction} and~\cite[Sections 6.7 and 2.4]{laurent2009sums}.

A central pillar of the mentioned method is the notion of atomic measures.

\begin{definition}[Finitely-atomic measure]
    Let $\bK \subset \R^n$ and $\mu \in \cM(\bK)$ be a measure on $\bK$. We say $\mu$ is (finitely-) atomic if there exist (finitely) many points $\vx_1,\ldots,\vx_l \in \bK$ and weights $a_1,\ldots,a_l> 0$ such that
    \begin{equation*}
        \mu = \sum\limits_{i = 1}^l a_i \delta_{\vx_i}
    \end{equation*}
    where $\delta_{\vx_i}$ denotes the Dirac measure in $\vx_i$. We call the points $\vx_1,\ldots,\vx_m \in \bK$ the atoms of $\mu$.
\end{definition}

The starting point for exact minimizer extraction is the following straightforward observation.

\begin{lem}\label{lem:AtomicMinimizer}
    Let $\mu \in \cM(\bK)$ be a minimizer of the measure formulation (\ref{eq:DualpopPos}) of the POP (\ref{eq:pop}). Assume $\mu$ is atomic with $\mu = \sum\limits_{i = 1}^l a_i \delta_{\vx_i}$ for $a_1,\ldots,a_l >0$. Then each atom $\vx_i$ is a minimizer of (\ref{eq:pop}).
\end{lem}

The statement of Lemma \ref{lem:AtomicMinimizer} is well-known, see for instance~\cite[Theorem 6.6]{lasserre2015introduction}. Because the proof is very short, we state it here.

\begin{proof}
    % First note that $\sum\limits_{i = 1}^l a_i = \int\limits \one \; \od \mu = \mu(\bK) = 1$ by feasibility of $\mu$ for (\ref{eq:DualpopPos}). Let $\vx^\star \in K$ be a minimizer of $f$, i.e. it holds $f(\vx^\star) = f^\star$. Consider the feasible measure $\mu^star := \delta_{\vx^\star} = \sum\limits_{i = 1}^l a_i \delta_{\vx_i} \in \cM(\bK)$. By minimality, we have $\int\limits f\; $
    
    % Thus, from $f(\vx) \geq f^*$ for all $\vx\in \bK$ -- including the atoms $\vx_1,\ldots,\vx_l \in \bK$ -- we have
    % \begin{eqnarray}\label{eq:AtomicMeasureMinimizer}
    %     f^\star = \left(\sum\limits_{i = 1}^l a_i\right) \cdot f^\star = \sum\limits_{i = 1}^l a_i \cdot f^\star \leq \sum\limits_{i = 1}^l a_i f(\vx_i) = \int\limits f \; \od \mu \leq f^\star
    % \end{eqnarray}
    % where the last inequality holds because $\mu$ is a minimizer of (\ref{eq:DualpopPos}). From (\ref{eq:AtomicMeasureMinimizer}) we get $f^\star = \sum\limits_{i = 1}^l a_i f(\vx_i)$. We conclude $f(\vx_i) = f^\star$, for all $i = 1,\ldots,l$, because $a_i>0$ and $f(\vx_i) \geq f^\star$ for all $i = 1,\ldots,l$.
    First note that $\sum\limits_{i = 1}^l a_i = \int\limits \one \; \od \mu = \mu(\bK) = 1$ by feasibility of $\mu$ for (\ref{eq:DualpopPos}). Thus, from $f(\vx) \geq f^*$ for all $\vx\in \bK$ -- including the atoms $\vx_1,\ldots,\vx_l \in \bK$ -- we have
    \begin{eqnarray}\label{eq:AtomicMeasureMinimizer}
        f^\star = \left(\sum\limits_{i = 1}^l a_i\right) \cdot f^\star = \sum\limits_{i = 1}^l a_i \cdot f^\star \leq \sum\limits_{i = 1}^l a_i f(\vx_i) = \int\limits f \; \od \mu \leq f^\star
    \end{eqnarray}
    where the last inequality holds because $\mu$ is a minimizer of (\ref{eq:DualpopPos}). We infer that both inequalities in (\ref{eq:AtomicMeasureMinimizer}) is an equality. For the first, this is the case if and only if for all $1\leq i \leq l$ it holds $f(\vx_i) = f^\star$ (because $f(\vx_i) \geq f^\star$ and all the weights $a_i$ are strictly positive and sum to $1$). We conclude $f(\vx_i) = f^\star$ and the statement.
\end{proof}

The previous lemma indicates that for obtaining minimizers of the POP (\ref{eq:pop}) we are interested in finitely-atomic minimizers $\mu \in \cM(\bK)$ of (\ref{eq:DualpopPos}). This is where flatness, see Definition \ref{def:flatness}, will again prove useful. The following theorem is a refined version of Theorem \ref{thm:flatness1}.

\begin{thm}[{\cite[Theorem 6.18]{laurent2009sums}}]\label{thm:KrepMeasure}
    Set $K = \cK(\vp)$ and $d\in \N$ with $d\geq r:= \max\limits_{i} \deg (p_i)$. Consider $\cQ_d(\vp)^*$ from (\ref{eq:def:QuadModDual}) and let $L \in \cQ_d(\vp)^*$ be $r$-flat. Then, there exist at most $\rank (M^{L})$ many atoms $\vx_1,\ldots,\vx_m \in \bK$, and weights $a_1,\ldots,a_m > 0$ such that $\mu := \sum\limits_{i = 1}^{m} a_i \delta_{\vx_i}  \in \cM(\bK)$ satisfies
    \vspace{-2mm}
    \begin{equation*}
        L(p) = \int\limits p \; \od \mu \quad \quad \text{for all } p \in \R[\bX]_{2d}.
    \end{equation*}
    Further, if $L$ is optimal for (\ref{eq:MomentHierarchy}), it holds $m^\star_d = f^\star$ and $\vx_1,\ldots,\vx_m$ are minimizers for the POP (\ref{eq:pop}).
\end{thm}

In order to turn Theorem \ref{thm:KrepMeasure} into an efficient method of constructing minimizers for the POP (\ref{eq:pop}), we need to discuss how the atoms $\vx_i$ of an atomic measure can be efficiently extracted from finitely many of its moments. This task is treated in~\cite{henrion2005detecting} where an efficient computational method is presented. We refer to the texts~\cite{henrion2005detecting,lasserre2008semidefinite,lasserre2015introduction,laurent2009sums} for the method and details. This leads to the following established procedure of minimizer computation~\cite[Theorem 10.2]{lasserre2015introduction}:
\begin{enumerate}
    \item Formulate the moment hierarchy for the POP (\ref{eq:pop}).
    \item Choose $d\in \N$ with $d\geq \max\{\deg (f), \deg (p_1),\ldots,\deg (p_m)\}$ and solve the moment-relaxation (\ref{eq:MomentHierarchy}) for optimal $L^\star_d$.
    \item Check if $L^\star_d$ is $r$-flat for $r = \max\limits_{i} \deg (p_i)$ if not, go to Step 2 and increase $d$.
    \item If $L^\star_d$ is $r$-flat, then follow the method from ~\cite{henrion2005detecting} (see also ~\cite[Algorithm 6.9]{lasserre2015introduction} and~\cite[Sections 6.7 and 2.4]{laurent2009sums}) to compute atoms of the representing measure $\mu$ from Theorem \ref{thm:KrepMeasure}.
\end{enumerate}
If this procedure terminates, i.e. if Step 4. is reached, then Theorem \ref{thm:KrepMeasure} guarantees that the optimal value $f^\star$ is found, and Lemma \ref{lem:AtomicMinimizer} implies that the extracted atoms $\vx$ are optimal points $\vx^\star$ for the POP (\ref{eq:pop}). However, in general, whether or not this procedure terminates cannot be answered in polynomial time unless $P = NP$, see~\cite{vargas2024hardness}. For this reason, we investigate another method for \textit{approximating} the minimizers from optimal solutions $L^\star_d$ that is applicable also when $L^\star_d$ is not flat.

\begin{rem}[Related methods]
    A method -- that applies as well for non-commutative polynomial optimization -- of (robustly) extracting minimizers can be found in~\cite{klep2018minimizer}. It should also be emphasized that~\cite{klep2018minimizer} investigates softening the concept of flatness and allowing for ``almost flat'' operators. In Section \ref{sec:alternatives}, we mention two different methods from~\cite{marx2021semi,lasserre2022christoffel} and~\cite{infusino2023intrinsic} for estimating the support of a measure by its moments.
\end{rem}

\subsection{Constructing candidate minimizers from linear moments}\label{Sec:CandidateMinimizers}

In the previous Section \ref{Sec:ExtractExactMinimizers} we described a method for extracting minimizers of (\ref{eq:pop}). This procedure applies when flatness occurs in the moment hierarchy. In this section, we complement the approach with a method that applies even also when flatness does not occur. As mentioned in Remark \ref{rem:ComplexityFeasiblePoint}, we should not even expect a method to efficiently provide feasible points $\vx\in\bK$ for a general POP. This is why, we examine a simple method for estimating minimizers with the risk that those approximative minimizers are not feasible for (\ref{eq:pop}). However, we will show that they lie increasingly close to true minimizers.

To outline the idea, let us assume the POP problem (\ref{eq:pop}) has a unique minimizer $\vx^\star \in\bK$, i.e. it holds $f(\vx^\star) = f^\star$ and $f(\vx) > f^\star$ for all $x\in \bK \setminus \{\vx^\star\}$. Then the (unique) optimal measure $\mu \in \cM(\bK)$ for the measure formulation (\ref{eq:DualpopPos}) of the POP (\ref{eq:pop}) is given by the Dirac measure $\mu:= \delta_{\vx^\star}$. From $\mu$ we can extract the point $\vx^\star = (x^\star_1,\ldots,x^\star_n)$ by taking linear moments: Let $q_i:= X_i \in \R[\bX]$, it holds
\begin{equation}\label{eq:MotivationLinearMoments}
    x^\star_i = \int\limits X_i \; \od \delta_{\vx^\star} = \int\limits q_i \; \od \mu \quad \text{for } i = 1,\ldots,n.
\end{equation}
In the same spirit, for $d\in \N$ and $L^\star_d$ a solution of the moment relaxation (\ref{eq:MomentHierarchy}), we estimate the minimizer $\vx^\star$ via
\begin{equation}\label{eq:MinEstimateExtraction}
    x^{(d,\star)}_i := L_d^\star(q_i) = L_d^\star(X_i). 
\end{equation}
In contrast to the exact case (\ref{eq:MotivationLinearMoments}), the estimators $\vx^{(d,\star)} := (x^{(d,\star)}_1,\ldots,x^{(d,\star)}_n)$ may be infeasible, i.e. it can happen that $\vx^{(d,\star)} \notin\bK$. For an illustrative purpose, we present the following simple example. 

\begin{example}\label{ex:EstimaterNonFeasible}
    Consider the set $K := \{0,1\}^2$. For $g_1(x_1,x_2):= x_1(1-x_1)$ and $g_2(x_1,x_2):= x_2(1-x_2)$ we have
    \begin{equation*}
       \bK = \{(x_1,x_2)\in \R^2: g_1(x_1,x_2) = g_2(x_1,x_2) = 0\} = \cK(g_1,-g_1,g_2,-g_2),
    \end{equation*}
    Let the cost be $f(x_1,x_2) := -x_1 - x_2 + x_1x_2$. By evaluating $f$ in the four points of $\bK$, we get
    \begin{equation}\label{eq:ExSimpleFStar}
        -1 = \min\limits_{\vx\in\bK} f(\vx).
    \end{equation}
    For $d \in \N$, the moment relaxation for the POP is given by
    \begin{eqnarray*}
        m_d^* := & \inf\limits_{L} & L(f)\\
	& \text{s.t.} & L \in \cQ_d(\vp)^*\notag\\
& & L(1) = 1. \notag
    \end{eqnarray*}
    Substituting $L$ by its pseudo-moments $y_\alpha := L(x^{\balpha})$ for $\balpha \in \N^2$ with $|\balpha|\leq d$ %, as in Remark \ref{rem:Pseudomoments},
    turns (\ref{eq:MomHierExample}) into an SDP. For $d = 2$, it reads
    \begin{eqnarray*}
    \begin{array}{ccc}
        m_2^* = & \inf\limits_{y_{10},y_{01},y_{20},y_{11},y_{02}} & -y_{10} - y_{01} + y_{11} \\
	& \text{s.t.} & \left(\begin{array}{ccc}
	     1 & y_{10} & y_{01} \\
	     y_{10} & y_{20} & y_{11}\\
          y_{01} & y_{11} & y_{02}
	\end{array}\right) \succeq 0\\
        & & y_{10} - y_{20} = 0\\
        & & y_{01} - y_{02} = 0
    \end{array}
    \end{eqnarray*}
    Substituting $y_{20} = y_{10}$ and $y_{01} = y_{02}$ results in the SDP
    \begin{eqnarray}\label{eq:MomHierExample}
    \begin{array}{ccc}
        m_2^* = & \inf\limits_{y_{10},y_{01},y_{11}} & -y_{10} - y_{01} + y_{11} \\
	& \text{s.t.} & \left(\begin{array}{ccc}
	     1 & y_{10} & y_{01} \\
	     y_{10} & y_{10} & y_{11}\\
          y_{01} & y_{11} & y_{01}
	\end{array}\right) \succeq 0
    \end{array}
    \end{eqnarray}
    Solving (\ref{eq:MomHierExample}) gives the optimal value $m_2^\star = -1.25$ with optimal point $\vy^\star= (y_{10}^\star,y^\star_{01},y^\star_{11}) = (\frac{3}{4},\frac{3}{4},\frac{3}{8})$. For the extracted point $\vx^{(2,\star)}$ from (\ref{eq:MinEstimateExtraction}) we have
    \begin{equation*}
        \vx^{(2,\star)} = (y^\star_{10},y^\star_{01}) = (\frac{3}{4},\frac{3}{4})\notin\bK.
    \end{equation*}
    We can also exclude the possibility of having chosen a ``wrong'' minimizer of (\ref{eq:MomHierExample}). That is, there is no other optimal point $\hat{\vy} = (\hat{y}_{10},\hat{y}_{01},\hat{y}_{11})$ for (\ref{eq:MomHierExample}) with $(\hat{y}_{10},\hat{y}_{01})\in\bK$. To verify this, let $\vy = (y_{10},y_{01},y_{11})$ be feasible for (\ref{eq:MomHierExample}) with $(y_{10},y_{01})\in\bK$; a case distinction for each of the four cases $(y_{10},y_{01}) \in \{0,1\}^2$ shows that the cost for $\vy$ is bounded from below by $-1 > -1.25$. Flatness arises at $d = 3$, i.e. the optimal value $-1 = m_3^\star$ and optimal points $(1,0),(0,1),(1,1)\in\bK$ are retrieved.
    %For $(y_{10},y_{01}) \in \{ (0,0),(1,0),(0,1)\}$ this follows immediately from the SDP. For $(y_{10},y_{01}) = (1,1)$ the only feasible point is $y = (1,1,1)$ with cost $-1$.
\end{example}

As we have seen in the above example, the point $\vx^{(d,\star)}$ from (\ref{eq:MinEstimateExtraction}) might not be feasible for the POP (\ref{eq:pop}). Hence, we should consider different indicators for the quality of $\vx^{(d,\star)}$ as a candidate minimizer. In this context, it is reasonable to ask how close $\vx^{(d,\star)}$ lies to $\vx^\star$ (the related question of how close $f(\vx^{(d,\star)})$ is to $f^\star$ will be treated in Proposition \ref{prop:FeasibiliteEstimatorUpperBound} for convex $f$). We can ask the same question differently:
\begin{center}
How close lie the linear pseudo-moments $L_d^\star(X_i)$ of $L_d$ to the linear moments of $\delta_{\vx^\star}$?
\end{center}
This motivates two natural generalizations: First, we can investigate higher-order moments. Second, we consider POPs (\ref{eq:pop}) with possibly multiple minimizers. For such POPs, we replace the measure $\delta_{\vx^\star}$ by a general minimizer $\mu \in \cM(\bK)$ of the measure formulation (\ref{eq:DualpopPos}) of the POP (\ref{eq:pop}). Thus, we ask:
\begin{eqnarray}\label{eq:QuestionMomentConvergence}
    & \text{How close lie the pseudo-moments } L_d^\star(\bX^{\balpha}) \text{ of } L_d \text{ to the moments } \int\limits \bX^{\balpha} \; \od \mu & \\
    & \text{ of a minimizer } \mu \text{ of (\ref{eq:DualpopPos})?}& 
\end{eqnarray}
The result~\cite[Theorem 12]{schweighofer2005optimization} gives a qualitative answer to the above question: As $d$ tends to infinity, the pseudo-moments of $L_d^\star$ lie arbitrarily close to the moments of some optimal measure for (\ref{eq:DualpopPos}). In case of a unique minimizer of (\ref{eq:pop}), this implies that $\vx^{(d,\star)}$ converges to the optimal point $\vx^\star$ as $d\rightarrow \infty$. In the following section, we present a quantitative version of this result. 

% \begin{rem}
%     Note that even though convergence of $\vx^{(d,\star)}$ to the minimizer $\vx^\star$ is shown in~\cite{schweighofer2005optimization}, the feasibility of $\title{\vx}^{(d)}$, i.e. whether or not it holds $\title{\vx}^{(d)} \in\bK$, is not answered. An argument for why this question remains difficult to solve is that finding feasible points for (arbitrary) semialgebraic sets is an intrinsically difficult problem.
% \end{rem}

\section{Convergence to optimal moments}\label{sec:MomentConvergence}

The convex structure of the measure formulation (\ref{eq:DualpopPos}) of the POP (\ref{eq:pop}) implies that the set of optimal points of (\ref{eq:DualpopPos}) is convex. More precisely, the set of optimal points for (\ref{eq:DualpopPos}) is given by
\begin{equation}\label{eq:OptMeasures}
    \{\mu \in \cM(\bS^\star) : \mu(\bK) = 1\} \quad \text{for} \quad \bS^\star := \{\vx \in\bK: f(\vx) = f^\star\}.
\end{equation}
Consequently, we would expect the pseudo-moments $L_d^\star(\bX^{\balpha})$ of an optimal point $L^\star_d$ for (\ref{eq:MomentHierarchy}) to be close to the moments of an optimal point $\mu \in \cM(\bS^\star)$ -- and not necessarily to the moments of a Dirac measure $\delta_{\vx}$ in an optimal point $\vx \in \bS^\star$. %We will show that we can bound the distance of the moments of an optimal point $L^\star_d$ for (\ref{eq:MomentHierarchy}) to the moments of an optimal measure $\mu \in \cM(\bS^\star)$ in terms of the cost $L_d^\star(f)$
Before we state our main result for the pseudo-moments, we first show the analog result for measures.

\begin{lem}\label{lem:AlmostOptimalMeasure}
    Let $\bK \subset [-1,1]^n$ be closed basic semialgebraic, $f\in \R[\bX]$, $r\in \N$ and $\mu \in \cM(\bK)$ be a probability measure with
    \begin{equation}\label{eq:AlmostOptMeasure}
        \int\limits f \; \od \mu \leq f^\star  + \delta
    \end{equation}
    for some $\delta \geq 0$. Then, there exists a measure $\mu_r \in \cM(\bS^\star)$ such that for all $\balpha \in \N^n$ with $\abs{\balpha} \leq r$ it holds
    \begin{equation*}
        \abs{\int\limits \bX^{\balpha} \; \od \mu -\int\limits \bX^{\balpha} \; \od \mu_r} \leq \abs{\balpha} C \delta^{\frac{1}{\mathcal{L}}}
    \end{equation*}
    for constants $C,\mathcal{L}\geq 0$ that depend only on $f$ and $\bK$ (and not on $r$). Further, $\mathcal{L}$ can be chosen with $\mathcal{L} \geq 1$.
\end{lem}

\begin{proof}
    Let $\balpha \in \N^n$ with $\abs{\balpha} \leq r$. %By scaling, we may assume that $\mu$ is a probability measure. By searching for a probability measure $\mu_{r}$, the claimed bound holds for $\abs{\balpha} = 0$.
    Without loss of generality we may assume that $f^\star = 0$, i.e. that $f\geq 0$ on $\bK$ and $f(\vx^\star) = 0$. Otherwise, we consider the function $\Tilde{f}:= f - f^\star$. In this proof we will use two main ingredients:
    \begin{enumerate}
    \begin{subequations}
    \item[a)] \textit{The \L{}ojasiewicz inequality}: We apply Theorem \ref{thm:loja} and Lemma \ref{lem:LojaGeq1} to the continuous functions $f$ and
    \begin{equation}\label{eq:gDist}
        g(\vx) := \dist(x,\bS^\star)
    \end{equation}
    That is, we get constants
    \begin{equation}\label{eq:L2Geq1}
       \c \geq 0 \quad \text{and} \quad \mathcal{L} \geq 1 
    \end{equation}
    with
    \begin{equation}\label{eq:LojafDist}
        g(\vx)^{\mathcal{L}} \leq \c \abs{f(\vx)} =  \c f(\vx) \quad \text{for all }\vx \in\bK. 
    \end{equation}
    That Theorem \ref{thm:loja} can be applied is assured by Lemma \ref{lem:DistSemialgebraic}, namely it states that $g$ is semialgebraic.
    %\green{Is taking $g_r(\vx) := \norm{x^\star - x}^r$ of interest?}
    \item[b)] \textit{Richter-Tchakaloff theorem}: By Proposition \ref{prop:Richter} applied to $E := \mathrm{span}\left(\{f\} \cup \{\bX^{\balpha}: \abs{\alpha \leq r}\right)$, we get an integer $l\leq \binom{n+r}{r} + 1$, points $\vx^{(1)},\ldots,\vx^{(l)} \in\bK$, and weights $a_1,\ldots,a_l \in (0,1]$ such that for the atomic measure $\nu \in \cM(\bK)$ with
    \begin{equation}\label{eq:DefNu}
        \nu := \sum\limits_{i = 1}^l a_i \delta_{\vx^{(i)}}
    \end{equation}
    it holds
    \begin{equation}\label{eq:Nu=Mu'OnE}
        \int\limits f \; \od \mu = \int\limits f \; \od \nu \quad \quad \text{and} \quad \quad \int\limits \bX^{\balpha} \; \od \mu = \int\limits \bX^{\balpha} \; \od \nu \quad \text{for all } \abs{\balpha}\leq r.
    \end{equation}
    \end{subequations}
    \end{enumerate}
    Now we construct $\mu_r \in\cM(\bS^\star)$. Motivated by $\nu$ from (\ref{eq:DefNu}) we set
    \begin{equation}\label{eq:UpperBoundDefMuMultipleMinimizers}
        \mu_{r} := \sum\limits_{i = 1}^l a_i \delta_{\vx^{(i,\star)}} \in \cM(\bS^\star) \quad \text{for } \vx^{i,\star} \in \argmin\left\{\norms{\vx - \vx^{(i)}}{2} : \vx \in \bS^\star\right\}.
    \end{equation}
    Integrating the \L{}ojasiewicz inequality (\ref{eq:LojafDist}) against $\nu$, using (\ref{eq:gDist}) and (\ref{eq:Nu=Mu'OnE}), gives
    \begin{eqnarray}\label{eq:IntLojaf}
         \c \int\limits f(\vx) \; \od \mu \overset{\text{(\ref{eq:Nu=Mu'OnE})}}{=} \c\int\limits f(\vx) \; \od \nu \geq \int\limits g(\vx)^{\mathcal{L}} \; \od \nu = \sum\limits_{i= 1}^l a_i g(\vx^{(i)})^{\mathcal{L}} \overset{\text{(\ref{eq:gDist})}}{=} \sum\limits_{i= 1}^n a_i \norms{\vx^{(i)} - \vx^{i,\star}}{2}^{\mathcal{L}}.
    \end{eqnarray}
    It remains to put the ingredients together -- we get
    \begin{equation}\label{eq:ConvRateMeasureMinimizer4}
        \arraycolsep=1.4pt\def\arraystretch{2.2}
        \begin{array}{ccll}
            \abs{\displaystyle\int\limits \bX^{\balpha} \; \od \mu - \displaystyle\int\limits \bX^{\balpha} \; \od \mu_r} & = & \abs{\displaystyle\int\limits \bX^{\balpha} \; \od \nu - \int\limits \bX^{\balpha} \; \od \mu_r} & \text{by (\ref{eq:Nu=Mu'OnE})}\\
            & = & \abs{\sum\limits_{i = 1}^l a_i (\vx^{(i)})^{\balpha} - \sum\limits_{i = 1}^l a_i (\vx^{(i,\star)})^{\balpha}} & \text{by definition of }\nu, \mu\\
            & \leq & \sum\limits_{i = 1}^l a_i \abs{(\vx^{(i)})^{\balpha} - (\vx^{(i,\star)})^{\balpha}}\\
            & \leq &  \abs{\balpha} \sum\limits_{i = 1}^l a_i \norms{\vx^{(i)} - \vx^{(i,\star)}}{2} & \text{by Lemma \ref{lem:LipschitzMonomial}}\\
            & \leq &  \abs{\balpha} \left(\sum\limits_{i = 1}^l a_i \norms{\vx^{(i)} - \vx^{(i,\star)}}{2}^{\mathcal{L}}\right)^{\frac{1}{\mathcal{L}}} \quad \quad \hspace{4mm}& \text{by Jensen's inequality}\\
            & \leq &  \abs{\balpha} \left(\c \displaystyle\int\limits f \; \; \od \nu\right)^{\frac{1}{\mathcal{L}}} & \text{by (\ref{eq:IntLojaf})}\\
            & = &  \abs{\balpha} \left(\c \displaystyle\int\limits f \; \; \od \mu\right)^{\frac{1}{\mathcal{L}}} & \text{by (\ref{eq:Nu=Mu'OnE})}\\
            & \leq &  \abs{\balpha} \c^{\frac{1}{\mathcal{L}}} \left(f^\star + \delta\right)^{\frac{1}{\mathcal{L}}}. & \text{by (\ref{eq:AlmostOptMeasure})}
        \end{array}
    \end{equation}
    By setting $C:= \c^\frac{1}{\mathcal{L}}$, the claim follows.% because we could assume $f^\star = 0$.
\end{proof}

Now, we turn to pseudo-moments, i.e. to feasible points $L \in \cQ_d(\vp)$ for (\ref{eq:MomentHierarchy}). To account for practical solutions, we work with ``$\varepsilon$-almost optimal'' solutions $L_d$ of (\ref{eq:MomentHierarchy}). That is, for given $\varepsilon_d\geq 0$, we consider feasible points $L_d$ of (\ref{eq:MomentHierarchy}) with
\begin{equation}\label{eq:AlmostOptL}
    L_d(f) \leq m_d^\star + \varepsilon_d.
\end{equation}
In~\cite[Theorem 12]{schweighofer2005optimization} it was shown that, with $\lim\limits_{d\rightarrow \infty}\varepsilon_d = 0$, the moments $L_d(\bX^{\balpha})$ get arbitrarily close to moments of optimal measures for (\ref{eq:DualpopPos}) as $d$ tends to infinity. The following Theorem is a quantitative version of this result.

% \begin{rem}[{Existing convergence result,~\cite[Theorem 12]{schweighofer2005optimization}}]
%     The heuristic argument is that for each $\alpha \in \N^n$, the sequence of pseuodo-moments $(L^\star_d(\bX^{\balpha}))_{d \in \N}$ converges to the moments $\int\limits \bX^{\balpha} \; \od \mu$ of a measure $\mu \in \cM(\bK)$. In particular, it holds $\int\limits f \; \od \mu = \lim\limits_{d\rightarrow \infty} L_d^\star(f) = f^\star$ where the last inequality follows from the convergence of the moment hierarchy.
% \end{rem}

\begin{thm}\label{thm:MomConvToOptMom}
    Let $m\in \N$, $p_1,\ldots,p_m \in \R[\bX]$ satisfy (\ref{eq:ArchimAndNormalisationLorenzo}) and $f \in \R[\bX]$. Further, let $r,d \in \N$ with $\deg (f) \leq r \leq  d \in \N$ and $L_d$ be a feasible point for (\ref{eq:MomentHierarchy}) satisfying (\ref{eq:AlmostOptL}). There exists a probability measure $\mu_{r,d} \in \cM(\bS^\star)$ such that for $\balpha \in \N^n$ with $\abs{\balpha} \leq r$ it holds
    \begin{equation*}
        \abs{L_d(\bX^{\balpha}) - \int\limits \bX^{\balpha} \; \od \mu_{r,d}} \leq \abs{\balpha}  C \cdot \left(\varepsilon_d^{\frac{1}{\text{\L}}} +\gamma(n,\vp)^{\frac{1}{2.5 n \text{\L}}} t^{\frac{12}{5}}\binom{n+t}{t} d^{-\frac{1}{2.5 n \text{\L}\cdot \text{\L}_2}}\right)
    \end{equation*}
    for $t:= \max\{r,\deg(f)\}$, constants $\text{\L}, \text{\L}_2 ,C \geq 0$ which depend only on $f$ with $\text{\L}_2 \geq 1$, and the constant $\gamma(n,\vp)$ from Theorem \ref{thm:OldBaldi}.
\end{thm}

\begin{proof}
    Let $\balpha \in \N^n$ with $\abs{\balpha} \leq r$. We may assume $\abs{\balpha} \geq 1$. For $\balpha = 0 \in \N^n$, due to feasibility, it holds $L_d(\bX^{\balpha}) = L_d(1) = 1 = (\vx^\star)^\vO = (\vx^\star)^{\balpha}$. Again, without loss of generality, we may assume that $f^\star = 0$, i.e. that $f\geq 0$ on $\bK$ and $f(\vx^\star) = 0$. In this proof, we will pair Lemma \ref{lem:AlmostOptimalMeasure} with the convergence rate for pseudo-moments from Theorem \ref{thm:BaldiMoment}. The strategy is as follows: Using Corollary \ref{cor:ConvRateHausDist} we find a measure $\mu'\in \cM(\bK)$ whose moments approximate the pseudo-moments $L_d(\bX^{\balpha})$ and for which $\int\limits f\; \od \mu$ is close to the cost $L_d(f)$. Then we apply Lemma \ref{lem:AlmostOptimalMeasure} to find a measure $\mu \in \cM(\bS^\star)$ that is ``close'' to $\mu'$. Consequently, also the moments of $\mu$ lie close to the pseudo-moments $L_d(\bX^{\balpha})$. In the rest of the proof, we specify this strategy and quantify the arguments.
    \begin{subequations}
    \begin{enumerate}
        \item \textit{Finding a measure $\mu' \in \cM(\bK)$ close to $L_d$:} By Corollary \ref{cor:ConvRateHausDist}, we can find a measure $\mu' \in \cM(\bK)$ with $\mu(\bK) = 1$ such that for some constant $\text{\L}$ and all $\abs{\balpha}\leq r$ it holds
        \begin{equation}\label{eq:ConvRateMeasureMinimizer1}
            \abs{\int\limits \bX^{\balpha} \; \od \mu' - L_d(\bX^{\balpha})} \leq 6 \gamma(n,\vp)^{\frac{1}{2.5n\text{\L}}} t^\frac{12}{5} \binom{n+t}{t} d^{-\frac{1}{2.5n\text{\L}}}  =: \eta_{1,d}
        \end{equation}
        and
        \begin{equation}\label{eq:ConvRateMeasureMinimizer2}
            \abs{\int\limits f \; \od \mu' - L_d(f)} \leq 6 \gamma(n,\vp)^{\frac{1}{2.5n\text{\L}}} t^\frac{12}{5} \binom{n+t}{t} d^{-\frac{1}{2.5n\text{\L}}} \norms{f}{\mathrm{coeff}} =: \eta_{2,d}
        \end{equation}
        with the constants \L{} and $\gamma(n,\vp)$ from Corollary \ref{cor:ConvRateHausDist} and the term $\norms{q}{\mathrm{coeff}} := \sum\limits_{\alpha} q_\alpha^2$ for polynomials $q = \sum\limits_{\alpha} q_\alpha \vx^{\balpha} \in \R[\bX]$ is the norm defined in (\ref{eq:L2NormCoefficients}). The bound in (\ref{eq:ConvRateMeasureMinimizer2}) is exactly the statement of Corollary \ref{cor:ConvRateHausDist} applied to $q = f$. Similarly, the bound in (\ref{eq:ConvRateMeasureMinimizer1}) follows by taking $q(\vx) := \vx^{\balpha}$ in Corollary \ref{cor:ConvRateHausDist} and noting that for this choice of $q$ we have $\norms{q}{\mathrm{coeff}} = 1$. Note, that we did not choose $t = \abs{\balpha}$ in (\ref{eq:ConvRateMeasureMinimizer1}) because that might have led to different measures $\mu$ in (\ref{eq:ConvRateMeasureMinimizer1}) and (\ref{eq:ConvRateMeasureMinimizer2}).
        \item \textit{Defining $\mu \in \cM(\bS^\star)$:} From (\ref{eq:ConvRateMeasureMinimizer2}) we infer
        \begin{equation*}
            \int\limits f \; \od \mu' \leq f^\star + \abs{L_d(f) - f^\star} + \abs{L_d(f) - \int\limits f \; \od \mu'} \leq f^\star + \varepsilon_d + \eta_{2,d}.
        \end{equation*}
        Setting $\delta := \varepsilon_d + \eta_{2,d} \geq 0$ in Lemma \ref{lem:AlmostOptimalMeasure}, we find a measure $\mu \in \cM(\bS^\star)$ satisfying
        \begin{equation*}
            \abs{\int\limits \bX^{\balpha} \; \od \mu' -\int\limits \bX^{\balpha} \; \od \mu} \leq  \abs{\balpha} C' \left(\varepsilon_d + \eta_{2,d}\right)^{\frac{1}{\text{\L}_2}}
        \end{equation*}
        for constants $C' \geq 0,\text{\L}_2\geq 1$ from Lemma \ref{lem:AlmostOptimalMeasure}. Sublinearity of the map $z \mapsto z^{\frac{1}{\text{\L}_2}}$ implies
        \begin{equation}\label{eq:ConvRateMeasureMinimizer3}
            \abs{\int\limits \bX^{\balpha} \; \od \mu' -\int\limits \bX^{\balpha} \; \od \mu} \leq  \abs{\balpha} C' \left(\varepsilon_d^{\frac{1}{\text{\L}_2}} + \eta_{2,d}^{\frac{1}{\text{\L}_2}}\right)
        \end{equation}
    \item \textit{Relating $\mu$ to $L_d$:} Combining    (\ref{eq:ConvRateMeasureMinimizer1}) with (\ref{eq:ConvRateMeasureMinimizer3}), we have
        \begin{eqnarray}\label{eq:ConvRateMeasureMinimizer5}
            \abs{\int\limits \bX^{\balpha} \; \od \mu - L_d(\bX^{\balpha})} & \leq & \abs{\int\limits \bX^{\balpha} \; \od \mu - \int\limits \bX^{\balpha} \; \od \mu'} + \abs{\int\limits \bX^{\balpha} \; \od \mu' - L_d(\bX^{\balpha})}\notag\\
            & \leq &  \abs{\balpha} C' \varepsilon_d^{\frac{1}{\text{\L}_2}} +  \abs{\balpha} C' \eta_{2,d}^{\frac{1}{\text{\L}_2}} + \eta_{1,d}.
        \end{eqnarray}
        The rest of the proof consists of (coarsely) bounding the terms $\abs{\balpha} C' \eta_{2,d}^{\frac{1}{\text{\L}_2}}$ and $\eta_{1,d}$ in (\ref{eq:ConvRateMeasureMinimizer5}).
    \item \textit{Computing bounds for (\ref{eq:ConvRateMeasureMinimizer5}):} We begin by defining a constant $C''$
    \begin{equation}\label{eq:DefConstantsC'}
        C'' := 6 C' \cdot\norms{f}{\mathrm{coeff}}^{\frac{1}{\text{\L}_2}}.
    \end{equation}
    Inserting $C''$ into (\ref{eq:ConvRateMeasureMinimizer2}) gives, note that $\gamma(n,\vp) \geq 1$,
    \begin{equation}\label{eq:ConvRateMeasureMinimizer6}
       C' \eta_{2,d}^{\frac{1}{\text{\L}_2}} \leq C''\cdot \gamma(n,\vp)^{\frac{1}{2.5 n \text{\L}}} t^{\frac{12}{5}}\binom{n+t}{t} d^{-\frac{1}{2.5 n \text{\L} \cdot \text{\L}_2}}.
    \end{equation}
    This concludes a desired bound for the second summand in (\ref{eq:ConvRateMeasureMinimizer5}). For the third summand in (\ref{eq:ConvRateMeasureMinimizer5}) we have, because of $\text{\L}_2\geq 1$,
    \begin{eqnarray}\label{eq:ConvRateMeasureMinimizer7}
        \eta_{1,d} & = & 6 \gamma(n,\vp)^{\frac{1}{2.5n\text{\L}}} t^\frac{12}{5} \binom{n+t}{t} d^{-\frac{1}{2.5n\text{\L}}} \leq 6 \gamma(n,\vp)^{\frac{1}{2.5 n \text{\L}}}  t^{\frac{12}{5}}\binom{n+t}{t} d^{-\frac{1}{2.5 n \text{\L}\cdot \text{\L}_2}}.
    \end{eqnarray}
    Inserting (\ref{eq:ConvRateMeasureMinimizer6}) and (\ref{eq:ConvRateMeasureMinimizer7}) into (\ref{eq:ConvRateMeasureMinimizer5}) gives
    \begin{eqnarray*}
        \abs{\int\limits \bX^{\balpha} \; \od \mu - L_d(\bX^{\balpha})} \hspace{-1mm}& \leq & \hspace{-1mm} \abs{\balpha} C' \varepsilon_d^{\frac{1}{\text{\L}_2}} + ( \abs{\balpha} \cdot C'' + 6)\gamma(n,\vp)^{\frac{1}{2.5 n \text{\L}}} t^{\frac{12}{5}}\binom{n+t}{t} d^{-\frac{1}{2.5 n \text{\L}\cdot \text{\L}_2}}.
    \end{eqnarray*}
    The claim follows by setting $C := \max\{C',C'' + 6\}$ and $\mu_{r,d} := \mu$. \qedhere
    \end{enumerate}
    \end{subequations}
\end{proof}

For optimal solutions $L_d^\star$ of (\ref{eq:MomentHierarchy}) we can state an asympotic convergence rate.

\begin{cor}\label{cor:MomConvRateOpt}
    Let the assumptions of Theorem \ref{thm:MomConvToOptMom} be satisfied. For $r,d\in \N$ and an optimal point $L^\star_d$ for (\ref{eq:MomentHierarchy}), there exists $\mu_{r,d} \in \cM(\bS^\star)$ such that for $\alpha \in \N^n$ with $\abs{\balpha} \leq r$ it holds
    \begin{equation*}
        \abs{L_d^\star(\bX^{\balpha}) - \int\limits \bX^{\balpha} \; \od \mu_{r,d}} \leq \abs{\balpha}  C \gamma(n,\vp)^{\frac{1}{2.5 n \text{\L}}} t^{\frac{12}{5}}\binom{n+t}{t} d^{-\frac{1}{2.5 n \text{\L}\cdot \text{\L}_2}} \in \underset{d\rightarrow \infty}{\mathcal{O}}\left( d^{-\frac{1}{2.5n \text{\L}\cdot\text{\L}_2}} \right)
    \end{equation*}
    with $t:= \max\{r,\deg(f)\}$ and the same constants $\text{\L}, \text{\L}_2 ,C \geq 0$ as in Theorem \ref{thm:MomConvToOptMom}.
\end{cor}
\begin{proof}
    We can set $\varepsilon_d = 0$ in Theorem \ref{thm:MomConvToOptMom}.
\end{proof}

We accompany Theorem \ref{thm:MomConvToOptMom} and Corollary \ref{cor:MomConvRateOpt} with a few remarks on the distinction from~\cite[Theorem 1.8]{baldi2021moment}, and a discussion on the appearing constants, effects of improved effective Positivstellensätze, and practical relevance.

\begin{rem}[{Comparison to~\cite[Theorem 1.8]{baldi2021moment}}]\label{rem:ComparisonToLorenzoMomentConvergence}
    There are two differences between Theorem \ref{thm:MomConvToOptMom} and~\cite[Theorem 1.8]{baldi2021moment} (which we state in Theorem \ref{thm:BaldiMoment}). The principal one is that~\cite[Theorem 1.8]{baldi2021moment} considers measures supported on the semialgebraic set $\bK$ while Theorem \ref{thm:MomConvToOptMom} further specifies their support to be contained in the set of optimal points $S^\star$ from (\ref{eq:OptMeasures}). The second difference is the additional factor $\text{\L}_2$ in the exponent of $d$.
\end{rem}

\begin{rem}[Discussion on the constant $C$ in Theorem \ref{thm:MomConvToOptMom}]
    The constant $C$ is essentially defined in (\ref{eq:DefConstantsC'}) and contains the \L{}ojasiewicz constant $\c_2$ and the \L{}ojasiewicz exponent $\text{\L}_2$ from (\ref{eq:LojafDist}), as well as the norms $\norm{f}, \norms{f}{\mathrm{coeff}}$ defined in (\ref{eq:NormFHypercube}) and (\ref{eq:L2NormCoefficients}). Particularly, the quantity $\c_2$ is difficult to access because it depends on the (unknown) minimizer $\vx^\star$ and the behavior of $f$ around it.
\end{rem}

\begin{rem}[Improved Positivstellensätze]\label{rem:ImprovedPositivstellensätze}
    Any improved version of the effective Positivstellensatz Theorem \ref{thm:OldBaldi} leads to an improved convergence bound in Theorem \ref{thm:MomConvToOptMom}. The reason is that stronger rates in the effective Positivstellensatz lead to stronger rates for the approximation of pseudo-moments in the crucial Theorem \ref{thm:BaldiMoment}, see the discussion in~\cite[Section 3.3]{baldi2024lojasiewicz}. %Those moment approximations from Theorem \ref{thm:BaldiMoment} entered centrally in (\ref{eq:ConvRateMeasureMinimizer1}) and (\ref{eq:ConvRateMeasureMinimizer2}).
    An example of improved effective Positivstellensatz is the work~\cite{baldi2024lojasiewicz}, which presents an effective version of Putinar's Positivstellensatz without the dimension $n$ appearing in the exponents in the rate (\ref{eq:poprate}). However, the resulting improvement of the moment approximation in Theorem \ref{thm:BaldiMoment} has not been made explicit, which is why we work with \cite{baldi2021moment}.
\end{rem}

In the above remark, we mention how improvements on general Positivstellensätze tighten the convergence bound in Theorem \ref{thm:MomConvToOptMom}. Analogously, specialized Positivstellensätze for specific sets $\bK$ imply stronger rates in Theorem \ref{thm:MomConvToOptMom}.

\begin{rem}[Specialized effective Positivstellensätze]\label{rem:specializedPositivstellensätze}
    Specialized effective Positivstellensätze are Positivstellensätze tailored to certain sets $\bK$, such as the unit ball~\cite{slot2022sum}, the sphere~\cite{fang2021quantum,schlosser2024specialized}, or the hypercube~\cite{baldi2023psatz}, etc. For these sets, refined effective Positivstellensätze with convergence rates of order $\mathcal{O}\left(d^{-2}\right)$, as $d$ tends to infinity, have been obtained. Such tightened convergence bounds carry over to tighter bounds in Theorem \ref{thm:MomConvToOptMom} by similar arguments as in Remark \ref{rem:ImprovedPositivstellensätze}. %In particular, for the mentioned sets, the statement of Corollary \ref{cor:ConvRateEstimator} strengthens to $\norms{\vx^\star - \vx^{(d,\star)}}{2} \in\mathcal{O}\left(d^{-2}\right)$.
    In~\cite{bach2023exponential}, even an exponential convergence rate in the effective Positivstellensatz is obtained on the hypercube for functions with local regularity properties around the minimizer $\vx^\star$. 
\end{rem}

\begin{rem}[Practical convergence rates]
    It is important to mention that the asymptotic analysis of the moment-SoS hierarchy for generalized moment problems might not transfer to practical applications. The reason is threefold. Firstly, current computational capacities restrict the computation of the moment-SoS hierarchy already for medium-sized problems to low-degree instances. Secondly, the constants in (\ref{thm:MomConvToOptMom}) can be large and hide the asymptotic behavior at lower levels of the hierarchy; further, as indicated by~\cite[Theorem 3 and 4]{slot2022sum}, the convergence might appear only at higher level in the hierarchy (see the condition $d \geq n\deg (f) \sqrt{\deg (f)}$ in the statements of~\cite[Theorem 3 and 4]{slot2022sum}). And thirdly, the conditioning of the moment-relaxation (\ref{eq:MomentHierarchy}) tends to worsen with increasing level $d$ in the hierarchy. In our setting, this could cause the term $\varepsilon_d$ to approach zero at a slow rate (or even not at all).
\end{rem}

In the next section, we transfer the bounds on convergent rates for the pseudo-moments to bounds on approximating (unique) minimizers of the POP (\ref{eq:pop}) by the estimators $\vx^{(d)}$ from Section \ref{Sec:CandidateMinimizers}.

\subsection{Convergence rate for the estimated minimizers}\label{Sec:ConvRateMinimizer}

Here we want to investigate how Theorem \ref{thm:MomConvToOptMom} can be used to analyze convergence of the candidate minimizers (\ref{eq:MinEstimateExtraction}) introduced in Section \ref{Sec:CandidateMinimizers}.
%We have mentioned in Section \ref{Sec:Flatness} that flatness is a generic property for the moment-SoS hierarchy, implies finite convergence, and allows for an exact extraction of minimizers. In Remark \ref{rem:ComplexityFeasiblePoint}, we have also noted that deciding whether (or at which level) the moment hierarchy enjoys flatness is computationally complex (see also the last paragraph in Section \ref{Sec:ExtractExactMinimizers}). For this reason, we considered a different method to estimate minimizers in Section \ref{Sec:CandidateMinimizers}. The convergence rate for this method is analyzed here.
To begin, we recall shortly the construction of these candidate minimizers.

\paragraph{Minimizer estimation from Section \ref{Sec:CandidateMinimizers}} Consider the POP (\ref{eq:pop}), i.e.
\begin{eqnarray*}
    f^* := \inf\limits_{\vx} & f(\vx) &\\
	\text{s.t.} &\vx \in \cK(\vp) & \notag
\end{eqnarray*}
Next, for given $d\in \N$ and dual cone $\cQ_d(\vp)^*$ given by (\ref{eq:def:QuadModDual}), we consider the moment-relaxation (\ref{eq:MomentHierarchy}), that is
\begin{eqnarray*}
    m_d^* := & \inf\limits_{L \in \R[\bX]_d^*} & L(f)\\
	& \text{s.t.} & L \in \cQ_d(\vp)^*\notag\\
& & L(1) = 1 \notag
\end{eqnarray*}
and obtain an almost optimal operator $L_d$ with, for given $\varepsilon_d\geq 0$,
\begin{equation*}
    L_d(f) \leq m_d^\star + \varepsilon_d.
\end{equation*}
We define a candidate point $\vx^{(d)} = (x^{(d)}_1,\ldots,x^{(d)}_n) \in \R^n$ as in (\ref{eq:MinEstimateExtraction}) via linear pseudo-moments by
\begin{equation*}
    x^{(d)}_i := L_d(X_i). 
\end{equation*}

We are interested in the question of how close $\vx^{(d)}$ lies to a minimizer $\vx^\star$ of the POP (\ref{eq:pop}). Before treating this question, we emphasize a fundamental problem in this context -- \textit{convexity}. Convexity in the measure formulation 
(\ref{eq:DualpopPos}) of the POP (\ref{eq:pop}) can cause ``mixing'' of different minimizers when only considering linear moments. We illustrate this in the following example.

\begin{example}[Multiple minimizer]\label{ex:MultipleMinimizer}
    Consider a POP with multiple minimizers $\vx^{\star,1},\ldots,x^{\star,s} \in\bK$. As mentioned in (\ref{eq:OptMeasures}), the set of optimal measures $\cM^\star$ for (\ref{eq:DualpopPos}) is given by
    \begin{equation*}
        \cM^\star = \left\{ \sum\limits_{i = 1}^s a_i \delta_{\vx^{\star,i}} : a_1,\ldots,a_s \geq 0, \sum\limits_{i= 1}^s a_i = 1\right\}.
    \end{equation*}
    For $L:=\sum\limits_{i = 1}^s a_i \delta_{\vx^{\star,i}} \in \cM^\star$ with $a_1,\ldots,a_s \geq 0$ and $\sum\limits_{i = 1}^s a_i = 1$ the vector $\Tilde{\vx} := \left(L(X_i)\right)_{i = 1,\ldots,n}$ of linear moments takes the form 
    \begin{equation}\label{eq:ConvexCombMinimizers}
        \Tilde{\vx} = \sum\limits_{i = 1}^s a_i \vx^{\star,i}.
    \end{equation}
    In other words, any convex combination of minimizers of the POP (\ref{eq:pop}) can be obtained as linear moments of an optimal measure for (\ref{eq:DualpopPos}). However, such a point (\ref{eq:ConvexCombMinimizers}) need not be optimal for the POP (\ref{eq:pop}). This is not an infinite-dimensional phenomenon -- it can also occur for the moment-relaxation (\ref{eq:MomentHierarchy}) at a finite level $d\in \N$, for instance when there are multiple minimizers and flatness holds.
    % Consider $K = [0,1] = \cK(p)$ for $p(\vx) := x(1-x))$ and $f(\vx) = x(1-x)$. We already have for $d = 2$ that $f = p \in \cQ_d(p)$. In particular, it holds $f_d^\star = m_d^\star = f^\star = 0$. The two minimizers are $\vx^{\star,0} = 0$ and $\vx^{\star,1} = 1$. This is recognized by the moment hierarchy as well; for $d = 2$ the moment hierarchy gives the optimal points
    % \begin{equation*}
    %     L^{\star,s}_2 = \left(\begin{array}{cc}
    %          1 & s \\
    %          s & s
    %     \end{array}\right) \quad \text{for} \quad s \in [0,1].
    % \end{equation*}
    % The two extremal points are $L^{\star,0}_2$ and $L^{\star,1}_2$ which correspond to the moments of the measures $\delta_{\vx^{\star,0}}$ and $\delta_{\vx^{\star,1}}$.
\end{example}

\paragraph{Case of unique minimizers}

%The Dirac measures in the optimal points are the extreme points of the set of optimal points of the moment formulation. Therefore, it might be generic that the optimal point has a single atom. This is not so clear, because the set of optimal points depends on $f$ already!

%At the same time, it is not so clear what happens at the SDPs. Is it clear that generically they should approach extremal points?

Example \ref{ex:MultipleMinimizer} examines the situation of multiple minimizers for the POP (\ref{eq:pop}) and the resulting complications for linear minimizer estimation. However, the situation changes significantly, when there is a unique minimizer $\vx^\star$ of (\ref{eq:pop}). In that case, also the optimal measure for the optimization problem (\ref{eq:DualpopPos}) is unique and given by $\delta_{\vx^\star}$. By Theorem \ref{thm:MomConvToOptMom}, this implies that the pseudo-moments of solutions $L_d$ to moment-relaxation (\ref{eq:MomentHierarchy}) approximate the moments of $\delta_{x^\star}$. This readily gives the following result.

%%%%%%% prepared until here 

\begin{cor}\label{cor:EstimatorRateAlmostOpt}
    Assume the conditions of Theorem \ref{thm:MomConvToOptMom} are satisfied. Further assume that there exists a unique minimizer $\vx^\star$ of the POP (\ref{eq:pop}). Then, with $r,d,L_d,\alpha$ as in Theorem \ref{thm:MomConvToOptMom}, it holds
    \begin{equation*}
        \abs{(\vx^\star)^{\balpha} - L_d(\bX^{\balpha})} \leq \abs{\balpha}  C \cdot \left(\varepsilon_d^{\frac{1}{\text{\L}}} +\gamma(n,\vp)^{\frac{1}{2.5 n \text{\L}}} t^{\frac{12}{5}}\binom{n+t}{t} d^{-\frac{1}{2.5 n \text{\L}\cdot \text{\L}_2}}\right)
    \end{equation*}
    with the same constants $t,C,\text{\L},\text{\L}_2, \gamma(n,\vp)$ from Theorem \ref{thm:MomConvToOptMom}.
\end{cor}

\begin{proof}
    We apply Theorem \ref{thm:MomConvToOptMom} and find a measure $\mu_{r,d}$ that is optimal for (\ref{eq:MomentHierarchy}) with
    \begin{equation*}
        \abs{\int\limits \bX^{\balpha} \; \od \mu_{r,d} - L_d(\bX^{\balpha})} \leq \abs{\balpha}  C \cdot \left(\varepsilon_d^{\frac{1}{\text{\L}}} +\gamma(n,\vp)^{\frac{1}{2.5 n \text{\L}}} t^{\frac{12}{5}}\binom{n+t}{t} d^{-\frac{1}{2.5 n \text{\L}\cdot \text{\L}_2}}\right).
    \end{equation*}
    By the assumption on the uniqueness of the minimizer $\vx^\star$ of (\ref{eq:pop}), the optimal measure for (\ref{eq:DualpopPos}) is unique and given by $ \delta_{\vx^\star}$, i.e. we have $\mu_{r,d} = \delta_{\vx^\star}$. Thus, its moments are given by
    \begin{equation*}
        \int\limits \bX^{\balpha} \; \od \mu_{r,d} = \int\limits \bX^{\balpha} \; \od \delta_{\vx^\star} = (\vx^\star)^{\balpha}.
    \end{equation*}
    The statement follows.
\end{proof}

The argument can be slightly refined when we are interested in the Euclidean distance between the estimator $\vx^{(d)}$ and the unique minimizer $\vx^\star$.

\begin{cor}\label{cor:ConvRateEstimator}
    Under the assumptions of Theorem \ref{thm:MomConvToOptMom}. Let $\vx^{(d)}$ be given by (\ref{eq:MinEstimateExtraction}). It holds
    \begin{equation*}
        \norms{\vx^{(d)} - \vx^\star}{2} \leq C \varepsilon_d^\frac{1}{\text{\L}_2} + (C + \sqrt{n}) \cdot \gamma(n,\vp)^{\frac{1}{2.5 n \text{\L}}} t^{\frac{12}{5}}\binom{n+t}{t} d^{-\frac{1}{2.5 n \text{\L} \cdot \text{\L}_2}}%\in \underset{d\rightarrow \infty}{\mathcal{O}}\left( d^{-\frac{1}{2.n\text{\L} \cdot \text{\L}_2}} \right)
    \end{equation*}
    with the same notation and constants $t,C,\text{\L},\text{\L}_2, \gamma(n,\vp)$ from Theorem \ref{thm:MomConvToOptMom}.
\end{cor}

\begin{proof}
    We could apply Theorem \ref{thm:MomConvToOptMom} with $t= \max\{1,\deg (f)\}$ and get
    \begin{eqnarray*}
        \norms{\vx^{(d)} - \vx^\star}{2}^2 = \sum\limits_{i = 1}^n \abs{x^{(d)}_i - x^\star_i}^2 \leq \sum\limits_{i = 1}^n\left( C \varepsilon_d^{\frac{1}{\text{\L}_2}} + C \cdot \gamma(n,\vp)^{\frac{1}{2.5 n \text{\L}}} t^{\frac{12}{5}}\binom{n+t}{t} d^{-\frac{1}{2.5 n \text{\L} \cdot \text{\L}_2}}\right)^2,
    \end{eqnarray*}
    i.e. $\norms{\vx^{(d)} - \vx^\star}{2} \leq \sqrt{n} \, C\left(\varepsilon_d^{\frac{1}{\text{\L}_2}} + \gamma(n,\vp)^{\frac{1}{2.5 n L}} t^{\frac{12}{5}}\binom{n+t}{t} d^{-\frac{1}{2.5 n L}}\right)$. Compared to the claimed rate, this induces an additional factor $\sqrt{n}$ for the term $\varepsilon_d^{\frac{1}{\text{\L}_2}}$. Instead, to get the claimed rate, we follow the proof of Theorem \ref{thm:MomConvToOptMom} and refine the bound in (\ref{eq:ConvRateMeasureMinimizer4}) by applying Lemma \ref{lem:VectorIntNormBound}. Therefore, we take a measure $\mu'$ as in (\ref{eq:ConvRateMeasureMinimizer1}) and (\ref{eq:ConvRateMeasureMinimizer2}) and apply Lemma \ref{lem:VectorIntNormBound} to the function $h(\vx) := \vx - \vx^\star\in \R^n$. This gives
    \begin{equation}\label{eq:IntNormBound2}
        \norms{\left(\int\limits X_i - x_i^\star\; \od \mu'\right)_{i = 1,\ldots,n}}{2} \overset{\text{Lemma \ref{lem:VectorIntNormBound}}}{\leq} \int\limits \norms{\bX - \vx^\star}{2}\; \od \mu'.
    \end{equation}
    We get the claimed rate from (\ref{eq:IntNormBound2}) and following the computations in the proofs of Lemma \ref{lem:AlmostOptimalMeasure} and Theorem \ref{thm:MomConvToOptMom}. The rest of the proof consists of adapting those computations.
    
    Using Jensen's inequality and \L{}ojasiewicz inequality (\ref{eq:LojafDist}) (note that in (\ref{eq:LojafDist}) the function $g(\vx) := \dist(\vx,\bS^\star)$ equals $g(\vx) =  \norms{\vx-\vx^\star}{2}$ in our current setting) -- we get, with the constants $\c$ and $\mathcal{L}$ from (\ref{eq:LojafDist}),
    \begin{equation}\label{eq:CorBoundNorm}
        \int\limits \norms{\bX - \vx^\star}{2}\; \od \mu' \leq \left(\int\limits \norms{\bX - \vx^\star}{2}^{L}\; \od \mu'\right)^\frac{1}{L} \overset{\text{(\ref{eq:ConvRateMeasureMinimizer4})}}{\leq} \c^\frac{1}{L} \left(\int\limits f\; \od \mu'\right)^\frac{1}{L} \leq \c^\frac{1}{L} \left(\int\limits f\; \od \mu'\right)^\frac{1}{L}.
    \end{equation}
    As in Step 4. in the proof of Theorem \ref{thm:MomConvToOptMom}, we infer $\left(\int\limits f\; \od \mu'\right)^\frac{1}{L}\leq \varepsilon_d^\frac{1}{L} + \eta_{2,d}^\frac{1}{L}$. Together with (\ref{eq:CorBoundNorm}), this gives 
    \begin{equation}\label{eq:BoundIntXf}
        \int\limits \norms{\bX - \vx^\star}{2}\; \od \mu' \leq \varepsilon_d^\frac{1}{L} + \eta_{2,d}^\frac{1}{L}
    \end{equation}
    Putting together, we get
    \begin{eqnarray*}
        \norms{\vx^{(d)} - \vx^\star}{2} & = & \norms{\left( L_d^\star(X_i) - x_i^\star\right)_{i = 1,\ldots,n}}{2}\notag\\
        & \leq & \norms{\left(L_d^\star(X_i) - \int\limits X_i\; \od \mu'\right)_{i = 1,\ldots,n}}{2} + \norms{\left(\int\limits X_i - x_i^\star\; \od \mu'\right)_{i = 1,\ldots,n}}{2}\notag\\
        & \overset{\text{(\ref{eq:IntNormBound2})}}{\leq} & \sqrt{\sum\limits_{i = 1}^n \underbrace{\abs{\int\limits X_i \; \od \mu' - L_d^\star(X_i)}^2}_{\leq \eta_{1,d} \text{ from (\ref{eq:ConvRateMeasureMinimizer1})}}} + \int\limits \norms{\bX - \vx^\star}{2}\; \od \mu'\\
        & \overset{\text{(\ref{eq:BoundIntXf})}}{\leq} & \sqrt{n} \, \eta_{1,d} + \c^\frac{1}{L}\left(\varepsilon_d^\frac{1}{L} + \eta_{2,d}^\frac{1}{L}\right)
    \end{eqnarray*}
    where used the term $\eta_{1,d}$ from (\ref{eq:ConvRateMeasureMinimizer1}). Defining the constant $C$ as in Step 4. in the proof of Theorem \ref{thm:MomConvToOptMom} concludes the claim (by the same coarse bounds (\ref{eq:ConvRateMeasureMinimizer6}), (\ref{eq:ConvRateMeasureMinimizer7})).
\end{proof}

\begin{rem}[Symmetry reduction and leveraging sparsity]
    When multiple minimizers arise resulting from symmetries, the number of minimizers can be reduced by symmetry reduction, see for instance~\cite{riener2013exploiting} where symmetry is leveraged for POP. In such cases, the minimizers may be reduced to a unique one, such as in~\cite{augier2023symmetry} where symmetry reduction allowed to reduce to a unique solution in some problem instances of optimal control problems tackled via the moment-SoS hierarchy.

    Another concept that aims at reducing the computational complexity of the moment-relaxation (\ref{eq:MomentHierarchy}) is sparsity, see for instance~\cite[Section 8]{lasserre2015introduction}. Our approach transfers directly to correlation sparsity for which convergence rates similar to the effective Putinar's Positivstellensatz from Theorem \ref{thm:OldBaldi} are available, see~\cite{korda2024convergence}.
\end{rem}

\section{Moment convergence for an upper bound moment-SoS hierarchy}\label{sec:UpperboundHierarchy}

For polynomial optimization problems (\ref{eq:pop}), the moment-SoS hierarchy from Definition \ref{def:lasserrehierarchy} and (\ref{eq:MomentHierarchy}) is not the only hierarchy based on SoS decompositions. The subject of this section is a different hierarchy -- an upper-bound hierarchy from\cite{lasserre2011new}. We investigate the convergence of its moments as in Section \ref{sec:MomentConvergence}.

\paragraph{An upper bound hierarchy}
In~\cite{lasserre2011new}, the following hierarchy was induced: For $d\in \N$ consider
\begin{eqnarray}\label{eq:UpperboundHierarchy}
    u_d^\star := & \inf\limits_{\substack{\sigma \in \Sigma\\ \deg (\sigma)\leq d}} & \int\limits f \cdot \sigma \; \od \mu\\
    & \mathrm{s.t.} & \int\limits \sigma \; \od \mu = 1.\notag
\end{eqnarray}
where $\mu$ is a fixed reference measure with support $\mathrm{supp} (\mu) = \bK$ whose moments are known. Since every polynomial $\sigma \in \Sigma$ is non-negative we get for any feasible $\sigma$ for (\ref{eq:UpperboundHierarchy})
\begin{equation*}
    \int\limits f \cdot \sigma \; \od \mu \geq \int\limits f^\star \cdot \sigma \; \od \mu = f^\star \int\limits \sigma \; \od \mu = f^\star.
\end{equation*}
Further, since the feasible set in (\ref{eq:UpperboundHierarchy}) is monotonically increasing with $d$, it follows
\begin{equation}\label{eq:UpperBoundHierarchyValues}
    u_d^\star \geq u_{d+1}^\star \geq f^\star.    
\end{equation}
The upper-bound hierarchy (\ref{eq:UpperboundHierarchy}) is motivated by the infinite-dimensional LP
\begin{eqnarray}\label{eq:UpperBoundLP}
    u^\star := & \inf\limits_{\sigma \in \Sigma} & \int\limits f \cdot \sigma \; \od \mu \\
    & \mathrm{s.t.} & \int\limits \sigma \; \od \mu = 1.\notag
\end{eqnarray}
By definition, we have $u^\star = \lim\limits_{d\rightarrow \infty} u_d^\star$, and further we have $u^\star = f^\star$ see~\cite{lasserre2011new}. The second equality can be verified by approximating the Dirac delta $\delta_{x^\star}$ in a minimizer $x^\star$ by a measure $\nu \in \cM(\bK)$ with $\od\nu = \sigma \; \od\mu$ for a ``needle'' density $\sigma \in \Sigma$. More generally, we can identify each $\sigma \in \Sigma$ with a measure $\nu_\sigma \in \cM(\bK)$ given by
\begin{equation}\label{eq:defnusigma}
    \od \nu_\sigma := \sigma \od \mu,
\end{equation}
that is, it holds
\begin{equation*}
    \int\limits g \; \od \nu_\sigma = \int\limits g \cdot \sigma \; \od \mu \quad \text{for all } g \in \R[\bX].
\end{equation*}
Using $\nu_\sigma$ we can reformulate the optimization problem (\ref{eq:UpperboundHierarchy}) as
\begin{eqnarray}\label{eq:UpperboundHierarchyMeasure}
    u_d^\star = & \inf\limits_{\nu \in \cM(\bK)} & \int\limits f \; \od \nu\\
    & \mathrm{s.t.} & \nu = \nu_\sigma \text{ for some } \sigma \in \Sigma \text{ with } \deg (\sigma)\leq d\notag\\
    & & \int\limits 1 \; \od \nu = 1.\notag
\end{eqnarray}
This observation allows us to view (\ref{eq:UpperboundHierarchy}) as a tightening of the measure formulation (\ref{eq:DualpopPos}) of the POP (\ref{eq:pop}), i.e. the feasible points for (\ref{eq:UpperboundHierarchy}) are measures. This is in contrast to the moment-relaxations (\ref{eq:MomentHierarchy}) which act on pseudo-moments. In the following paragraph, we leverage some advantages of working with measures in (\ref{eq:UpperboundHierarchyMeasure}) compared to working with pseudo-moments in (\ref{eq:MomentHierarchy}).

\paragraph{A convergence rate for the moments in the upper bound hierarchy}
Before we give a corollary of Lemma \ref{lem:AlmostOptimalMeasure}, we recall the set $\bS^\star := \{\vx \in \bK: f(\vx) = f^\star\}$ of minimizers of $f$ in $\bK$.

\begin{cor}\label{cor:UpperboundApprox}
    Let $\bK\subset [-1,1]^n$ be closed, $r\in \N$ and $\sigma \in \Sigma$ be feasible for (\ref{eq:UpperBoundLP}) with $\int\limits f \cdot \sigma \; \od \mu \leq f^\star  + \delta$ for some $\delta \geq 0$. Then there exists a probability measure $\mu_{r}\in \cM(\bS^\star)$ such that for all $\balpha \in \N^n$ with $\abs{\balpha}\leq r$ it holds
    \begin{equation*}
        \abs{\int\limits \bX^{\balpha} \cdot \sigma\; \od \mu - \int\limits \bX^{\balpha}\; \od \mu_{r}} \leq  \abs{\balpha} C\delta^\frac{1}{L}
    \end{equation*}
    for constants $C\geq 0, L\geq 1$ that depend only on $f$ and $\bK$ (but not on $r$).
\end{cor}

\begin{proof}
    The statement follows from applying Lemma \ref{lem:AlmostOptimalMeasure} to the measure $\nu_\sigma$ defined in (\ref{eq:defnusigma}).
\end{proof}

In the above Corollary, we can substitute $\delta$ by whichever convergence rate is available for the hierarchy of upper bounds (\ref{eq:UpperboundHierarchy}). For convex bodies the following convergence rate has been established in~\cite[Theorem 3]{de2017convergence}: There exists $d_0\in \N$ such that for all $d_0\leq d\in \N$
\begin{equation}\label{eq:UpperBoundConvRate}
    u^\star_d - f^\star \leq \frac{C}{\sqrt{d}}
\end{equation}
for a constant $C\geq 0$ depending only on $f$ and $\bK$. As in previous sections, to account for practical imprecision, we allow for almost optimal solutions of (\ref{eq:UpperboundHierarchy}). That is, for $d\in \N$ let $\sigma_d \in \Sigma$ feasible for (\ref{eq:UpperboundHierarchy}) with given precision $\varepsilon_d\geq 0$, i.e.
\begin{equation}\label{eq:UpperBoundAlmostOpt}
    \int\limits f \cdot \sigma_d \; \od \mu \leq u_d^\star + \varepsilon_d.
\end{equation}

\begin{cor}\label{cor:UpperConvBody}
    Let $f\in \R[\bX]$. Assume that $\bK \subset [-1,1]^n$ has non-empty interior and is a convex closed basic semialgebraic set. For $d\in \N$, let $\varepsilon_d \geq 0$ and $\sigma_d \in \Sigma$ be feasible for (\ref{eq:UpperboundHierarchy}) and satisfy (\ref{eq:UpperBoundAlmostOpt}). Then, there exists $d_0\in \N$ such that for $r,d\in \N$ with $d\geq d_0$ there exists a probability measure $\mu_{d,r} \in \cM(\bS^\star)$ with
    \begin{equation*}
        \abs{\int\limits \bX^{\balpha} \cdot \sigma_d\; \od \mu - \int\limits \bX^{\balpha}\; \od \mu_{d,r}} \leq  \abs{\balpha} C \left(\frac{1}{\sqrt{d}} + \varepsilon_d\right)^\frac{1}{\mathcal{L}}
    \end{equation*}
    for some constants $C\geq 0,\mathcal{L} \geq 1$ depending only on $f$ and $\bK$.
\end{cor}

\begin{proof}
    Let $r,d,\varepsilon_d,\sigma_d$ as in the statement. By~\cite[Theorem 3]{de2017convergence}, there exists $d_0 \in \N$ and a constant $C'$ such that for all $d\geq d_0$ it holds
    \begin{equation*}
        \int\limits f \; \od \nu_\sigma \leq u_d^\star + \varepsilon_d \overset{\text{(\ref{eq:UpperBoundConvRate})}}{\leq} f^\star + \frac{C'}{\sqrt{d}} + \varepsilon_d.
    \end{equation*}
    We can apply Corollary \ref{cor:UpperboundApprox} with $\delta :=\frac{C'}{\sqrt{d}} + \varepsilon_d$ and get a probability measure $\mu_{d,r}\in \cM(\bS^\star)$ with
    \begin{equation*}
        \abs{\int\limits \bX^{\balpha} \cdot \sigma\; \od \mu - \int\limits \bX^{\balpha}\; \od \mu_{r}} \leq  \abs{\balpha} C''\left(\frac{C'}{\sqrt{d}} + \varepsilon_d\right)^\frac{1}{\mathcal{L}}
    \end{equation*}
    for a constant $C''\geq 0$ and $\mathcal{L}\geq 1$. By setting $C:= C''\cdot\max\{1,C'\}$ the claim follows. 
\end{proof}

\begin{rem}
    The convergence rate (\ref{eq:UpperBoundConvRate}) applies to full-dimensional convex basic semialgebraic sets, however, convergence rates for the upper bound hierarchy (\ref{eq:UpperboundHierarchy}) have been established for several other (classes of) sets as well. We refer to~\cite{laurent2024convergence} for a survey on convergence rates for the moment hierarchy (\ref{eq:MomentHierarchy}) and the upper bound hierarchy (\ref{eq:UpperboundHierarchy}).
\end{rem}

\paragraph{Convergence rate for estimated minimizers}

As in Section \ref{Sec:ConvRateMinimizer}, we infer a convergence rate for an estimator for the minimizer of (\ref{eq:pop}) via linear moments. That is, for $d\in \N$ and $\sigma \in \Sigma$ feasible for (\ref{eq:UpperboundHierarchy}), we define the point $\check{\vx}^{(\sigma)} = (\check{\vx}^{(\sigma)}_1,\ldots,\check{\vx}^{(\sigma)}_n)\in \R^n$ by
\begin{equation}\label{eq:EstimatorUpperBound}
    \check{\vx}^{(\sigma)}_i := \int\limits X_i \cdot \sigma \; \od \mu \overset{\text{(\ref{eq:defnusigma})}}{=} \int\limits X_i \; \od \nu_\sigma \quad \text{for } i = 1,\ldots,n.
\end{equation}
By feasibility of $\sigma$ for (\ref{eq:UpperboundHierarchy}), the measure $\nu_\sigma$ from (\ref{eq:defnusigma}) is a probability measure. This implies that the estimator $\check{\vx}^{(\sigma)}$ lies in the convex hull $\mathrm{conv}(\bK)$ of $\bK$.

\begin{prop}\label{prop:FeasibiliteEstimatorUpperBound}
    Let $d\in \N$ and $\sigma \in \Sigma$ be feasible for (\ref{eq:UpperboundHierarchy}). Then $\check{\vx}^{(\sigma)}\in \mathrm{conv}(\bK)$. If, in addition, $f$ is convex on $\mathrm{conv}(\bK)$ then it also holds
    \begin{equation}\label{eq:UpperBoundEstimatorRate}
        f(\check{\vx}^{(\sigma)}) \leq \int\limits f \cdot \sigma \; \od \mu.
    \end{equation}
\end{prop}

\begin{proof}
    Let $\nu := \nu_{\sigma} \in \cM(\bK)$ be the probability measure for $\nu_{\sigma}$ defined as in (\ref{eq:defnusigma}). Consider the vector-valued integral $\displaystyle\int\limits \bX \; \od \nu \in \R^n$. By definition of $\check{\vx}^{(\sigma)}$ we have $\int\limits \bX \; \od \nu = \check{\vx}^{(\sigma)}$. Since $\nu$ is a probability measure, we also have $\displaystyle\int\limits \bX \; \od \nu \in \overline{\mathrm{conv}(\bK)}$. The latter statement is a standard result for vector-valued integrals, see for instance \cite[Section II.2. Corollary 8]{diestel1977vector}. This concludes the proof.%Here, we provide a short proof based on the Richter-Tchakaloff theorem, Proposition \ref{prop:Richter}. By Proposition \ref{prop:Richter}, there exist points $\vz_1,\ldots,\vz_n \in \bK$ and weights $a_1,\ldots,a_n \in [0,1]$ with $\sum\limits_{j = 1}^n a_j = 1$ such that
    % \begin{equation*}
    %     \int\limits \bX \; \od \nu = \left(\int\limits X_i \; \od \nu\right)_{i = 1,\ldots,n} = \sum\limits_{j = 1}^n a_j \vz_j \in \mathrm{conv}\{\vz_1,\ldots,\vz_n\} \subset \mathrm{conv}(\bK).
    % \end{equation*}
    % To prove the second claim of the statement, let $f$ be convex on $\mathrm{conv}(\bK)$. We have
    % \begin{eqnarray*}
    %     f(\check{\vx}^{(\sigma)}) & \overset{\text{(\ref{eq:UpperBoundVecInt})}}{=} & f\left(\int\limits \bX \; \od \nu \right) \overset{\text{Jensen's inequality}}{\leq} \int\limits f \; \od \nu \overset{\text{(\ref{eq:defnusigma})}}{=} \int\limits f \cdot \sigma \; \od \mu.
    % \end{eqnarray*}
\end{proof}

\begin{rem}
    Proposition \ref{prop:FeasibiliteEstimatorUpperBound} states that under additional convexity assumptions the cost of the $\check{\vx}^{(\sigma)}$ is at most the cost of $\sigma$ in the upper bound hierarchy (\ref{eq:UpperboundHierarchy}). Hence, the quality of solutions $\sigma \in \Sigma$ for (\ref{eq:UpperboundHierarchy}) immediately translates to the quality of the obtained estimator $\check{\vx}^{(\sigma)}$. In particular, for convex $f$, any convergence rate for the upper bound hierarchy (\ref{eq:UpperboundHierarchy}) implies at least the same rate for the costs $f(\check{\vx}^{(\sigma)})$.
\end{rem}

\begin{rem}[Verifying convexity]
    Despite that certifying convexity of a polynomial is NP-hard, \cite[Theorem 13.8]{lasserre2015introduction}, there are necessary and sufficient SoS certificates for convexity \cite[Theorem 13.9]{lasserre2015introduction}. Consequently, in view of Corollary \ref{cor:FeasibilityEstimator}, convexity of the cost $f$ on $\mathrm{conv}(K)$ can be checked in combination with \cite[Theorem 13.27]{lasserre2015introduction} (under the assumption that $K$ satisfies the \textit{Putinar bounded degree representation}, see \cite[Definition 13.25]{lasserre2015introduction}).
\end{rem}

As a consequence of Proposition \ref{prop:FeasibiliteEstimatorUpperBound} and Corollary \ref{cor:UpperConvBody}, we infer the following corollary on how close $\check{\vx}^{(\sigma)}$ lies to the unique minimizer $\vx^\star$.

\begin{cor}\label{cor:FeasibilityEstimator}
    Assume that $\bK$ has non-empty interior and is a convex closed basic semialgebraic set. Let $d\in \N$ and $\sigma$ be feasible for (\ref{eq:UpperboundHierarchy}) satisfying (\ref{eq:UpperBoundAlmostOpt}). Then $\check{\vx}^{(\sigma)} \in \bK$. Further, if the POP (\ref{eq:pop}) has a unique minimizer $\vx^\star$, there exists $d_0\in \N$ such that for all $d_0\geq d \in \N$
    \begin{equation*}
        \norms{\vx^\star - \check{\vx}^{(\sigma)}}{2} \leq C\left(\frac{1}{\sqrt{d}} + \varepsilon_d\right)^\frac{1}{\mathcal{L}}
    \end{equation*}
    for a constants $C\geq 0, \mathcal{L} \geq 1$ depending only on $f$ and $\bK$. 
\end{cor}

\begin{proof}
    The first part of the statement follows from the first statement in Proposition \ref{prop:FeasibiliteEstimatorUpperBound}. For the second part, note that, by assumption, we have $\bS^\star = \{x^\star\}$. Therefore, $\delta_{\vx^\star}$ is the only probability measure in $\cM(\bS^\star)$. We conclude the statement by similar arguments as in the proof of Corollary \ref{cor:ConvRateEstimator}.% to refine the bound from Corollary \ref{cor:UpperConvBody} for the vector of linear moments $\check{\vx}^{(\sigma)} = \int\limits \bX \; \od \nu_\sigma$ for $\nu_\sigma$ defined as in (\ref{eq:defnusigma}).
\end{proof}

Because convexity played an important role in this section, we want to give some concluding remarks on convexity for polynomial optimization. This includes relating our above analysis to results from \cite[Section 13]{lasserre2015introduction}.

\paragraph{Comparision to \cite[Section 13]{lasserre2015introduction}}
The text \cite[Section 13]{lasserre2015introduction} presents several notions of convexity for polynomial optimization and its consequences for the moment relaxation (\ref{eq:MomentHierarchy}). We base our comparison to~\cite[Section 13]{lasserre2015introduction} on the following three categories: notion of convexity, finite convergence of the moment relaxation (\ref{eq:MomentHierarchy}), and minimizer estimation via (\ref{eq:MinEstimateExtraction}).

We begin the discussion by emphasizing that the upper-bound hierarchy (\ref{eq:UpperboundHierarchy}) is based on measures $\nu_\sigma \in \cM(\bK)$, see (\ref{eq:defnusigma}), while the moment-relaxations (\ref{eq:MomentHierarchy}) act on pseudo-moments operators $L \in \cQ_{d}(\vp)^*$. Consequently, generalizing from measures to pseudo-moments required strengthening the notion of convexity. Emerged has SoS convexity:
\begin{equation}\label{eq:SoSConvex}
    f \text{ is called SoS convex if} \quad D^2f(x) = L(x)^TL(x) \text{ for some } L\in \R[x]^{n\times n}.
\end{equation}

\begin{enumerate}
    \item \textit{Notion of convexity:} SoS convexity implies convexity and therefore is more restrictive than using convexity only. Many -- but not all! -- of the results in \cite[Section 13]{lasserre2015introduction} require SoS convexity (\ref{eq:SoSConvex}). Hence, when we work with the upper bound hierarchy (\ref{eq:UpperboundHierarchy}) we require fewer conditions on the cost $f$ (and the defining polynomials $p$).
    \item \textit{Convergence rates:} In \cite[Theorem 13.32]{lasserre2015introduction} finite convergence with a simple degree bound is given under the assumption of SoS convexity in the cost $f$ and the defining polynomials $p_1,\ldots,p_m$. We do not match this result in Corollary \ref{cor:FeasibilityEstimator} because the upper bound hierarchy does not provide finite convergence except in degenerate cases.
    \item \textit{Minimizer estimation:} In \cite[Theorem 13.32]{lasserre2015introduction} not only shows finite convergence but also exact minimizer extraction via (\ref{eq:MinEstimateExtraction}) under SoS convexity. However, when SoS convexity is replaced by strict convexity, then still finite convergence holds but neither a degree bound nor feasibility of the candidate minimizer (\ref{eq:MinEstimateExtraction}) are available anymore in the current results \cite[Theorem 13.33]{lasserre2015introduction}. Therefore, in this situation, Corollary \ref{cor:FeasibilityEstimator}, provides stronger guarantees.
\end{enumerate}

We want to conclude this section with an analog result to Proposition \ref{prop:FeasibiliteEstimatorUpperBound} for the moment relaxations (\ref{eq:MomentHierarchy}). The following Corollary is a direct consequence of two convexity arguments from \cite[Section 13]{lasserre2015introduction} and the effective Positivstellensatz Theorem \ref{thm:OldBaldi}.

\begin{cor}\label{cor:ValConfMomRelax}
    Let $\bK = \cK(\vp) \subset [-1,1]^n$ satisfy the assumptions from Theorem \ref{thm:OldBaldi} and, additionally, the P-BDR property\footnote{The \textit{Putinar bounded degree property} (P-BDR) of order $d_0\in \N$ holds for $\bK$ if for any affine form $q:\R^n \rightarrow \R$ with $q \geq 0$ on $\bK$ we have $q \in \cQ_{d_0}(\vp)$.} of order $d_0\in \N$. Let $d_0 \leq d \in \N$. Consider the moment-relaxation (\ref{eq:MomentHierarchy}) with optimal value $m^\star_d$ and the candidate minimizer $x^{(d,\star)}$ from (\ref{eq:MinEstimateExtraction}). Then $x^{(d,\star)}\mathrm{conv}(\bK)$. Further, if $f$ is SoS convex (see (\ref{eq:SoSConvex})), then it also holds
    \begin{equation*}
        f^\star - f(x^{(d,\star})) \leq f^\star - m^\star_d \in \cO\left( d^{- 1 / 2.5 n \text{\L}}\right).
    \end{equation*}
\end{cor}

\begin{proof}
    The point $x^{(d,\star)}$ belongs to $\mathrm{conv}(\bK)$ by \cite[Theorem 13.27]{lasserre2015introduction}. This shows the first statement. For the second, by \cite[Theorem 13.21]{lasserre2015introduction}, it holds $f\left(x^{(d,\star)}\right) \leq m_d^\star$, and hence $f^\star - f(x^{(d,\star})) \leq f^\star - m_d^\star$. The claimed asymptotic rate is the statement of Corollary \ref{cor:poprate}.
\end{proof}

\section{Discussion on alternative approaches for minimizer extraction}\label{sec:alternatives}

In this work, we concentrate on estimating minimizers of the POP (\ref{eq:pop}) using the approach described in Section \ref{Sec:CandidateMinimizers}. As mentioned in Example \ref{ex:MultipleMinimizer}, the proposed approximation of minimizers of the POP (\ref{eq:pop}) is limited to instances with a unique minimizer. We wish to highlight two promising directions that can be found in~\cite{marx2021semi,lasserre2022christoffel} and~\cite{infusino2023intrinsic}, which likewise build upon (pseudo-) moments but are not restricted to POPs with a unique minimizers.

To motivate them, we recall that, as stated in (\ref{eq:OptMeasures}), the set of all optimal measures $\mu$ for the measure formulation (\ref{eq:DualpopPos}) are exactly the probability measures supported on the set of minimizers $\bS^\star$ of the POP (\ref{eq:pop}). %As we have seen in Theorem \ref{thm:MomConvToOptMom} and Corollary \ref{cor:UpperConvBody}, solving the moment relaxation (\ref{eq:MomentHierarchy}) or the moment tightening (\ref{eq:UpperboundHierarchy}), we access (approximatively) the moments of such optimal measures.
Consequently, extracting minimizers for the POP (\ref{eq:pop}) is closely related to identifying the support of a measure from its moments. In this section we want to emphasize two such methods from~\cite{marx2021semi,lasserre2022christoffel} and~\cite{infusino2023intrinsic}. %In our view, both these directions hold significant potential for addressing the problem of approximating minimizers of POPs from solutions of the moment-SoS hierarchy.
%Our goal in this section is \textit{not} to provide a concise quantitative analysis of minimizer extraction via those two methods -- but rather

The first method that we will discuss is based on the Christoffel-Darboux kernel which recently drew more and more attention in polynomial optimization and beyond, see for instance~\cite{lasserre2022christoffel,slot2022sum,marx2021semi}.

\paragraph{Christoffel-Darboux kernel methods}
We follow mostly~\cite{lasserre2022christoffel,marx2021semi}; note that the notation for $d$ and $n$ in ~\cite{lasserre2022christoffel} is interchanged with our use of $d,n$.

We begin by shortly recalling the Christoffel-Darboux kernel. Let $\mu$ be a measure on $\R^n$ with compact support $\bK := \supp(\mu)$. To fix notation, for $d\in \N$ we set $r(n,d):= \binom{n+d}{d}$ to be the dimension of $\R[\bX]_d$ and $M^{\mu,d} = (M^{\mu,d}_{\balpha,\bbeta})_{\abs{\balpha},\abs{\bbeta} \leq d} \in \R^{r(n,d) \times r(n,d)}$ to be the truncated moment matrix given by $M^{\mu,d}_{\balpha,\bbeta} := \int\limits \bX^{\balpha + \beta} \; \od \mu$. 

To avoid technical conditions, we assume for the rest of this section the (usual) condition that $\bK$ is compact and its interior $\mathrm{int}\ \bK$ is non-empty. For the interesting and important case of measures with singular support, we refer to~\cite[Sections 5 and 7.3.2]{lasserre2022christoffel}.

\begin{rem}
    A consequence of $\supp (\mu)$ having non-empty interior is that we can equip $\R[\bX]$ with the inner product $\langle\cdot,\cdot\rangle:\R[\bX]\times \R[\bX] \rightarrow \R$ defined by
    \begin{equation*}\label{eq:L2MuProduct}
        \langle g,h\rangle := \int\limits  g\cdot h\; \od \mu
    \end{equation*}
    In other words, we interpret $\R[\bX]$ as a subspace of $\mathrm{L}^2(\mu)$ the square-integrable measurable functions. The condition that $\supp (\mu)$ has non-empty interior guarantees that $\langle\cdot,\cdot\rangle$ is positive definite. That is, $\langle g,g\rangle > 0$ for all $g \in \R[\bX]\setminus \{0\}$, or, said differently, the moment matrix $M^{\mu,d}$ is invertible for all $d\in \N$. In particular, the finite dimensional spaces $\R[\bX]_d$ equipped with $\langle\cdot,\cdot\rangle$ are Hilbert spaces.
\end{rem}

The Christoffel-Darboux kernel is defined as follows.

% \begin{definition}[Christoffel-Darboux kernel]
%     For $d\in \N$ let $P_1,\ldots,P_{r(n,d)}$ be an orthonormal basis of $\R[\bX]_d$ with respect to the inner product (\ref{eq:L2MuProduct}). The Christoffel-Darboux kernel $K_{\mu,d}:\R^n \rightarrow \R^n \rightarrow \R$ is defined by
%     \begin{equation*}
%         K_{\mu,d}(\vx,\vy) := \sum\limits_{i = 1}^{r(n,d)} P_i(\vx) P_i(\vy).
%     \end{equation*}
% \end{definition}

\begin{definition}[Christoffel-Darboux kernel]
    Let $d\in \N$ and $M^{\mu,d}$ be the moment matrix of $\mu$. The Christoffel-Darboux kernel $K_{\mu,d}:\R^n \rightarrow \R^n \rightarrow \R$ is defined by
    \begin{equation*}
        K_{\mu,d}(\vx,\vy) := \vv(\vx)^T \left(M^{\mu,d}\right)^{-1}\vv(\vy)
    \end{equation*}
    for the moment vector $\vv:\R^n \rightarrow \R^{r(n,d)}$ with $\vv(\vz) := (\vz^{\balpha})_{\abs{\balpha} \leq d}$.
\end{definition}

%The Christoffel-Darboux kernel has many remarkable properties and gives important insights into, among others, approximation theory and quadrature rules, see~\cite{simon2008christoffel}. However,
Among the various applications of the Christoffel-Darboux kernel, we focus on a specific one -- the approximation of the support $\supp (\mu)$ of the measure $\mu$ given only (some of) its moments (approximately). % We recall two results from~\cite{lasserre2019empirical}, which can also be found in the book~\cite[Lemma 4.3.1 and 4.3.2]{lasserre2022christoffel}.
% \begin{lem}[{\cite[Lemma 4.3.1]{lasserre2022christoffel}}]
%     Let $\vx \notin \bK$ and $\delta := \dist(\vx,\bK)> 0$. Then
%     \begin{equation*}
%         \frac{K_{\mu,d}(\vx,\vx)}{r(n,d)} \geq 2^{\frac{\delta d}{\delta + \mathrm{diam} (\bK)} - 3} d^{-n} \left(\frac{n}{e}\right)^n e^{-\frac{n^2}{d}}.
%     \end{equation*}
% \end{lem}
% The above lemma states that for $\vx \notin \bK$ the Christoffel-Darboux kernel evaluated on points $(\vx,\vx)$ decays at a rate of $d^{-n}$ on as $d$ tends to infinity for . This criterion can be used to approximate $\bK$~\cite{lasserre2019empirical} by
% \begin{equation}\label{eq:SupportEstimCDKernel}
%     \bK_{d,C,n'}:=\left\{ \vx \in \R^n: K_{\mu,d}(\vx,\vx) < Cd^{-n'}) \right\}
% \end{equation}
% where $C\geq 0$ and $n' \geq n$ are well-chosen constants. The following lemma is leveraged to ensure that the set $\bK_{d,C,n'}$ contains $\bK$ (for large enough $d$).
% \begin{lem}[{\cite[Lemma 4.3.2]{lasserre2022christoffel}}]
%     Assume that $\bK$ is the closure of an open set with smooth boundary and assume $\od \mu = q \; \od \lambda\big|_{\bK}$ with $q \geq c > 0$ for some $c> 0$.
%     Let $\vx \in \mathrm{int} \ \bK$ and $\delta := \dist(\vx,\partial \bK)> 0$. Then
%     \begin{equation*}
%         \frac{K_{\mu,d}(\vx,\vx)}{r(n,d)} \leq 2 \frac{\mathrm{Vol}(\bK)}{c \delta^n s_n} (1+n)^3
%     \end{equation*}
%     where $s_n$ is the surface area of the unit sphere in $n+1$ dimensions. 
% \end{lem}
Following~\cite{marx2021semi}, for $d\in \N$ the support of $\mu$ can be approximated by the set $\bS_d$ given by
\begin{equation}\label{eq:DefCDApproxSupport}
    \bS_d := \left\{ \vx \in \R^n : K_{\mu,d}(\vx,\vx) < s_d \right\}
\end{equation}
where
\begin{equation*}
    s_d := \frac{1-\alpha}{16} \frac{e^{2r}d^r}{(3r)^{2r}}
\end{equation*}
for given $\alpha \in [0,1)$, $r\in \N$ with $r> d$. Under certain regularity assumptions on $\mu$ (see~\cite[Assumption 1]{marx2021semi}), the set $\bS_d$ converges to $\supp(\mu)$ in Hausdorff distance as $d$ tends to infinity, see~\cite[Theorem 5]{marx2021semi}. When working with operators $L \in \cQ_d(\vp)^*$ (see \ref{eq:def:QuadModDual}) for the definition of $\cQ_d(\vp)^*$) instead of measures $\mu$, a natural extension is to consider the corresponding function
\begin{equation*}
    K_{L,d}(\vx,\vy) := \vv(\vx)^T \left(M^{L}\right)^{-1}\vv(\vy)
\end{equation*}
where $M^L$ is the pseudo-moment matrix of $L$ defined as in (\ref{eq:MLBilinearForm}). Consequently, a candidate approximation of the ``support of $L$'' is then
\begin{equation*}
    \bS_{L,d} := \left\{ \vx \in \R^n : K_{L,d}(\vx,\vx) < s_d \right\}.
\end{equation*}
This approach has been proven useful, among others, for approximating graphs of solutions of hyperbolic partial differential equations~\cite{marx2021semi} or outlier detection~\cite{lasserre2022christoffel}. However, transferring their analysis to our application experiences some limitations which we address in the following remark.

\begin{rem}[Limitations]\label{rem:CDKernelLimitations}
    We divide the current limitations for the application of Christoffel-Darboux analysis to minimizer estimation into two categories. One is the restriction to measures with a certain regularity, and the other is computational aspects.
    \begin{enumerate}
        \item \textit{Regularity conditions:} The analysis for the Christoffel-Darboux kernel has mostly been investigating highly regular measures~\cite[Assumption 3.7]{lasserre2022christoffel},~\cite[Assumption 1]{marx2021semi}  -- more precisely, the considered measures $\mu$ often are assumed to have support with non-empty interior and to have a density with respect to the Lebesgue measure. We, on the other hand, are interested in measures with support on the variety $S^\star = \{\vx \in \bK: f(\vx) = f^\star\}$. Thus the above-mentioned regularity conditions are not satisfied.
        \item  \textit{Computational aspects:} Computing the set $\bS_d$ from (\ref{eq:DefCDApproxSupport}) requires to compute the points $\vx\in \R^n$ with $K_{\mu,d}(\vx,\vx) < s_d$. However, as we have seen at the example of polynomial optimization, computing sublevel sets of a (generic) polynomial is a complex task. Nevertheless, compared to the general situation, the polynomial $\vx \mapsto K_{\mu,d}(\vx,\vx)$ carries much structure that should be exploited, see for instance~\cite[Section 5.2.2]{de2002approximation}.
    \end{enumerate}
\end{rem}

There exist promising perspectives to overcome the limitations stated in the previous remark. We see these as interesting directions for future work.
    
\begin{rem}[Perspectives]
    We want to emphasize three promising perspectives to leverage Christoffel-Darboux kernel analysis to overcome (partially) the limitations mentioned in Remark \ref{rem:CDKernelLimitations}.
    \begin{enumerate}
        \item Measures with singular support: As noted in Remark \ref{rem:CDKernelLimitations}, for our application, the analysis of the Christoffel-Darboux kernel for measures supported on algebraic varieties is of interest. Work in this direction can be found in~\cite[Sections 5 and 7.3.2]{lasserre2022christoffel}. The special case of including atomic perturbations has been explored in~\cite[Section 24]{simon2008christoffel}.
        \item Christoffel-Darboux kernel for pseudo-moments: Solutions of the moment relaxation (\ref{eq:MomentHierarchy}) give rise to pseudo-moments and hence do not need to correspond to a measure $\mu$. In the work~\cite{marx2021semi} the estimation of the set $\bS_d$ in (\ref{eq:DefCDApproxSupport}) from pseudo-moments is investigated, see~\cite[Lemma 1 and Theorem 5]{marx2021semi}. Together with approximation results for pseudo-moments, such as Theorem \ref{thm:BaldiMoment}, this opens a pathway to quantitative Christoffel-Darboux analysis based on pseudo-moments.
        %\item Link between SOS multipliers and pseudo-moments in the moment-SoS hierarchy (\ref{def:lasserrehierarchy}) and (\ref{eq:MomentHierarchy}): The Christoffel-Darboux kernel is linked to sum-of-squares multipliers~\cite[Section 7.2.2]{lasserre2022christoffel} and therefore provides an additional link between the primal and the dual.
        \item Exploiting structure of the Christoffel-Darboux kernel: The Christoffel-Darboux polynomial $\vx \mapsto K_{\mu,d}(\vx,\vx)$ is a SoS polynomial. Furthermore, it is related to sum-of-squares multipliers for the SoS tightening (\ref{eq:LasserreHierarchy}) see ~\cite[Section 7.2.2]{lasserre2022christoffel} and~\cite{de2019approximate}. This provides an additional strong (computational) link between the SoS tightening (\ref{def:lasserrehierarchy}) and the moment relaxation (\ref{eq:MomentHierarchy}).
    \end{enumerate}
\end{rem}

% \begin{rem}[Duality between minimizer extraction and SoS representation]
%     In this work, we investigate ways of how solutions of the moment relaxation (\ref{eq:MomentHierarchy}) give rise to minimizers of (\ref{eq:pop}). The SoS tightening (\ref{eq:LasserreHierarchy}) does not share this advantageous property, but provides an SoS decomposition of $f -f_d^\star$, a feature not evident in the moment hierarchy. Recently, another connection between solutions of (regularised) moment hierarchies and optimal SoS multipliers for SoS tightenings has been established in~\cite{de2019approximate} based on the Christoffel-Darboux kernel.
% \end{rem}

\paragraph{A power method based on~\cite{infusino2023intrinsic}}
By~\cite[Theorem 1.2]{infusino2023intrinsic}, the support of a measure $\mu$ is given by
\begin{equation}\label{eq:SupportPowerInfusino}
    \supp (\mu) = \left\{ \vx \in \R^n: \abs{q(\vx)} \leq \sup\limits_{n\in \N} \left(\int\limits q^{2n}\; \od \mu \right)^{\frac{1}{2n}} \text{ for all } q\in \R[\bX]\right\}.
\end{equation}
In particular, for $\mathbb{\mathcal{F}}\subset \R[\bX]$ and $d\in \N \cup \{\infty\}$ we get an approximation $\bK_{\mathbb{\mathcal{F}},d}$ of $\supp(\mu)$ by
\begin{equation*}
    \bK_{\mathbb{\mathcal{F}},d}(\mu) := \left\{ \vx \in \R^n: \abs{q(\vx)} \leq \sup\limits_{\substack{n\in \N\\ n\cdot \deg(q)\leq d}} \left(\int\limits q^{2n}\; \od \mu \right)^{\frac{1}{2n}} \text{ for all } q\in\mathbb{\mathcal{F}} \right\}.
\end{equation*}
For $d = \infty$, the set $\bK_{\mathbb{\mathcal{F}},d}(\mu)$ contains the support of $\mu$. For $d< \infty$ none of the two sets $\bK_{\mathbb{\mathcal{F}},d}(\mu)$ and $\supp (\mu)$ is contained in the other in general. When, instead of a measure $\mu$, a linear map $L:\R[\bX]_{2d}\rightarrow \R$ is given, we consider the set
\begin{equation*}
    \bK_{\mathbb{\mathcal{F}},d}(L) := \left\{ \vx \in \R^n: \abs{q(\vx)} \leq \sup\limits_{\substack{n\in \N\\ n \cdot\deg (q)\leq d}} \left(L(q^{2n})\right)^{\frac{1}{2n}} \text{ for all } q\in\mathbb{\mathcal{F}} \right\}.
\end{equation*}
%$d = \infty$ is clear. For $d< \infty$ consider $\mu = \frac{1}{2} \delta_0 + \frac{1}{2}\delta_1$ and $\mathbb{\mathcal{F}} = \{q := X\}$. It holds $\supp(\mu) = \{0,1\}$ and  $\bK_{\mathbb{\mathcal{F}},d} = [-\left(\frac{1}{2}\right)^{\frac{1}{2d}},\left(\frac{1}{2}\right)^{\frac{1}{2d}}]$. 
The computation of $\bK_{\mathbb{\mathcal{F}},d}(L)$ is based solely on the (pseudo-) moments $L(\vx^{\balpha})$ for $|\balpha|\leq 2d$ and thus relates closely to our framework. However, efficient computation and quantitative analysis of the convergence of $\bK_{\mathbb{\mathcal{F}},d}(\mu)$ and $\bK_{\mathbb{\mathcal{F}},d}(L)$ as $d\rightarrow \infty$ are challenging and interesting tasks. To the best of our knowledge, currently there are no efficient methods available this task. Likewise, exploring whether the set $\mathbb{\mathcal{F}}\subset \R[\bX]$ and the integer $d \in \N$ can be chosen cleverly remains an intriguing task.

\section{Conclusion}\label{sec:Conclusion}
In this text, we give a first quantitative convergence analysis for the solutions $L_d^\star$ of the moment-SoS hierarchy towards optimal solutions $\mu^\star$ of the measure formulation (\ref{eq:DualpopPos}). This complements qualitative results in~\cite{schweighofer2005optimization} by quantitative results on the moment-SoS hierarchy from~\cite{baldi2021moment}. By borrowing \L{}yapunov inequality arguments from~\cite{baldi2021moment} we show that optimal measures $\mu^\star$ approximate the operators $L_d^\star$ by the same rate (up to a constant factor) as general measures do, see Theorem \ref{thm:MomConvToOptMom} and Remark \ref{rem:ComparisonToLorenzoMomentConvergence}. For polynomial optimization problems with unique minimizers, this results in a polynomial convergence rate for established linear estimators for optimal points of the POP.

Our analysis extends also to the so-called upper bound hierarchy, see Section \ref{sec:UpperboundHierarchy}. This hierarchy works directly with measures instead of pseudo-moments, which we leverage in an interplay with convexity properties of the underlying POP. We present that, when the underlying POP enjoys certain convexity properties, the estimated minimizer is feasible and a priori bounds on its cost can be computed from the upper-bound hierarchy, see Proposition \ref{prop:FeasibiliteEstimatorUpperBound}.

Our analysis is compatible with future improvements in effective Positivstellensätze, in the sense that refinements in the effective Positivstellenätze immediately strengthen the convergence bounds presented in this text, see Remark \ref{rem:ImprovedPositivstellensätze}. Consequently, specialized Positivstellensätze for specific sets, such as the sphere, the unit ball, or the hypercube~\cite{slot2022sum,fang2021sum,schlosser2024specialized,baldi2023psatz,bach2023exponential}, also prove beneficial in our framework, see Remark \ref{rem:specializedPositivstellensätze}.

Finally, in Section \ref{sec:alternatives}, we mention two other methods for estimating minimizers from solutions of moment-SoS relaxations/tightenings. One of these methods is based on the Christoffel-Darboux kernel which recently showed more and more promising applications in a wide variety of fields. The other is a power method allowing for outer approximations of the support of a measure. For both these methods, we discuss future perspectives and current limitations.

Another future direction could include the extension of this work to different SoS hierarchies, such as~\cite{wang2021tssos,ahmadi2019dsos,tran2024convergence}, or to non-commutative polynomial optimization  as in~\cite{burgdorf2016optimization}.

\section{Acknowledgment}
I want to thank Matteo Tacchi-Bénard for his fresh and insightful comments on the text and our conversations prior to and during the writing. I am thankful to Lorenzo Baldi for his comments on the flat extension and his detailed remarks on the convergence rates in the moment-SoS hierarchy. I want to thank Didier Henrion, Markus Schweighofer, and Saroj Chhatoi for their helpful discussions.

\section*{Appendix}
\appendix

\section{Approximation of pseudo-moments by truncated moment sequences}\label{Sec:AppendixPseudomoments}

\begin{subequations}
Throughout this section, we consider $p_1,\ldots,p_m \in \R[\bX]$ satisfying (\ref{eq:ArchimAndNormalisationLorenzo}) and the set $\bK$ denotes $\bK := \cK(\vp) = \{\vx \in \R^n : p_1(\vx)\geq 0, \ldots, p_m(\vx)\geq 0\}$.

In this section we want to state the result~\cite[Theorem 1.8]{baldi2021moment} in Theorem \ref{thm:BaldiMoment}. To do so, we first need to introduce the necessary notations from~\cite{baldi2021moment}.

\paragraph{Truncated measures}
 Each measure $\mu \in \cM(\bK)$ can be identified with a linear form $L_\mu:\R[\bX] \rightarrow \R$ via
 \begin{equation*}
     L_\mu(q):= \int\limits q \; \od \mu.
 \end{equation*}
 By non-negativity of $\mu$ it holds for all $q\in \cQ(\vp)$
 \begin{equation*}
     L_\mu(q) = \int\limits_K q \; \od \mu \geq 0
 \end{equation*}
In other words $L_\mu\big|_{\R[\bX]_d} \in \cQ_p(\vp)^*$ for all $d\in \N$. In coherence with~\cite{baldi2021moment}, we denote by $\cM^{(1)}(\bK)^t$ the set of ``probability measures truncated at degree $t \in \N$'', i.e.
 \begin{equation*}
     \cM ^{(1)}(\bK)^t := \{ L_\mu\big|_{\R[\bX]_d} : \mu \in \cM(\bK), \mu(\bK) = 1 \}
 \end{equation*}
 where the first upper index only indicates that the first moment $L_\mu (1)$ equals $1$.

 \begin{lem}\label{lem:M1KtClosed}
     For compact $\bK$ and $t\in \N$, the set $\cM ^{(1)}(\bK)^t$ is compact (with respect to $\norms{\cdot}{\mathrm{op}}$).
 \end{lem}

 \begin{proof}
    Because $\bK$ is compact, by the Banach-Alaoglu theorem, the set $\cP(\bK) := \{\mu \in \cM(\bK): \mu(\bK) = 1\}$ of probability measures is compact with respect to the weak$^*$ topology. For each $t \in \N$, the set $\cM ^{(1)}(\bK)^t$ is the image of the map $T:\cP \rightarrow \cM ^{(1)}(\bK)^t$ with $\mu \mapsto L_\mu\big|_{\R[\bX]_t}$. Because the image of a compact set under a continuous map is compact, we conclude the statement by showing that $T$ is continuous with respect to the weak$^*$ topology on $\cM(\bK)$. Therefore, let $(\mu_l)_{l\in \N} \subset \cP$ be a weak$^*$ converging sequence with limit $\mu\in \cP$, i.e. for all $q\in \R[\bX]$ it holds
    \begin{equation*}\label{eq:LimitMomentSequence}
        \int\limits q \; \od \mu_{l} \rightarrow  \int\limits q \; \od \mu \quad \text{as } j\rightarrow \infty.
    \end{equation*}
    In particular, by testing only with $q\in \R[\bX]_t$, we get
    \begin{equation*}\label{eq:LimitMomentSequence2}
        (T\mu_l)(q) = \left(L_\mu\big|_{\R[\bX]_t}\right)(q) (L_{\mu_l})(q) = \int\limits q \; \od \mu_{l} \rightarrow  \int\limits q \; \od \mu = L_\mu (q) = (T\mu)(q) \quad \text{as } j\rightarrow \infty.
    \end{equation*}
    This shows that $T$ is continuous with respect to the weak$^*$ topologies on $\cM(\bK)$ and $\R[\bX]_t$. Since $\R[\bX]_t^*$ is finite-dimensional, the weak$^*$ topology and any norm topology on $\R[\bX]_t^*$ coincide. Hence, the $T$ is continuous with respect to the weak$^*$ star topology on $\cM(\bK)$ and the $\norms{\cdot}{\mathrm{op}}$ topology on $\R[\bX]_t$. This is what remained to be shown.
    %  Let $t\in \N$. Let $(L_{\mu_l})_{l\in \N} \subset \cM ^{(1)}(\bK)^t$, for $\mu_l \in \cM(\bK)$ with $\mu_l(\bK)= 1$, be a converging sequence in $\R[\bX]_t^*$ with limit $L:\R[\bX]_t\rightarrow \R$, i.e. it holds for all $q\in \R[\bX]_t$ that
    %  \begin{equation}\label{eq:LimitMomentSequence}
    %      \int\limits q \; \od \mu_{l} = L_{\mu_l}(q) \rightarrow L(q) \quad \text{as } j\rightarrow \infty.
    %  \end{equation}
    %  Since the measures $\mu_l$ belong to the weak$^*$-compact set of probability measures, there exists a weak$^*$ convergent subsequence $(\mu_{l_j})_{j\in \N}$ and a probability measures $\mu$ such that for all $q \in \R[\bX]$ it holds
    % \begin{equation*}
    %     L_{\mu_{l_j}}(q) = \int\limits q \; \od \mu_{l_j} \rightarrow \int\limits q \; \od \mu = L_\mu(q) \quad \text{as } j\rightarrow \infty.
    % \end{equation*}
    % Together with (\ref{eq:LimitMomentSequence}) this shows $L = L_\mu\big|_{\R[\bX]_t}$, and hence that $\cM ^{(1)}(\bK)^t$ is closed.
 \end{proof}

 \paragraph{Pseudo-moments}
 Analog to the set $\cM ^{(1)}(\bK)^t$ of truncated probability measures, we define for $d,t\in \N$ with $t\leq d$ the truncation $L\big|_{\R[\bX]_t}$ of operators $L\in \cQ_d(\vp)^*$ with $L(1) = 1$, i.e.
 \begin{equation*}
     \cL _d^{(1)}(\vp)^t := \{ L\big|_{\R[\bX]_t}: L\in \cQ_d(\vp)^*, L(1) = 1 \}.
 \end{equation*}
 By identifying an operator $L \in \cQ_d(\vp)^*$ with the family of its pseudo-moments $L\left(\bX^{\balpha}\right)$ for $\balpha\in \N^n$ with $\abs{\balpha}\leq d$, we also refer to $\cL _d^{(1)}(\vp)^t$ as a set of pseudo-moments.

 \paragraph{Hausdorff distance}
 Both sets $\cL _d^{(1)}(\vp)^t$ and $\cM ^{(1)}(\bK)^t$ are subsets of the dual space $\R[\bX]_t^*$ of $\R[\bX]_t$. On $\R[\bX]_t$ we choose the norm $\norms{\cdot}{\mathrm{coeff}}$ defined as the Euclidean norm of the coefficients of a polynomial with respect to the standard monomial basis, i.e.
 \begin{equation}\label{eq:L2NormCoefficients}
    \norms{q}{\mathrm{coeff}}^2 := \sum\limits_{\alpha} q_\alpha^2 \quad \quad \text{for } q = \sum\limits_{\alpha} q_\alpha \vx^{\balpha} \in \R[\bX]_t.
 \end{equation}
 This equips $\R[\bX]_t^*$ with the dual norm, i.e. the induced operator norm
 \begin{equation}\label{eq:DualNorm}
     \norms{L}{\mathrm{op}} := \sup\limits_{\norms{q}{\mathrm{coeff}} \leq 1} |L(q)| = \sup\limits_{q \neq 0} \frac{|L(q)|}{\norms{q}{\mathrm{coeff}}}.
 \end{equation}
 The Hausdorff distance $d_H(A_1,A_2)$ (with respect to $\norms{\cdot}{\mathrm{op}}$) between two sets $A_1,A_2 \subset \R[\bX]^*_t$ is defined by
\begin{equation*}
    d_H(A_1,A_2) := \inf \{ \varepsilon > 0: B_\varepsilon(A_1) \supset A_2, B_\varepsilon(A_2)\supset A_1\}
\end{equation*}
where $B_\varepsilon(C):= \{c+x: c\in C,\vx \in B_\varepsilon(0)\}$ and $B_\varepsilon(0)$ denotes the ball centered at $0$ with radius $\varepsilon$ w.r.t the norm $\norms{\cdot}{\mathrm{op}}$ from (\ref{eq:DualNorm}).

\paragraph{Convergence rate for pseudo-moments}

To state the result~\cite[Theorem 1.8]{baldi2021moment} on the convergence rate for $d_H(\cM ^{(1)}(\bK)^t, \cL _d^{(1)}(\vp)^t)$ as $d\rightarrow \infty$, we need to introduce the following integer $d_0 \in \N$, see~\cite[Lemma 3.6]{baldi2021moment},
\begin{equation}\label{eq:HausDistConvd0}
    d_0 := \min \{k \in \N: 1-p_i \in \cQ_k(\vp) \; \text{ for } i = 1,\ldots,m\} < \infty.
\end{equation}

\begin{thm}[{\cite[Theorem 1.8]{baldi2021moment}}]\label{thm:BaldiMoment}
    Let $n\geq 2$, $m\in \N$ and $p_1,\ldots,p_m \in \R[\bX]$ such that (\ref{eq:ArchimAndNormalisationLorenzo}) holds. Let $d_0\in \N$ be given by (\ref{eq:HausDistConvd0}). Set $K = \cK(\vp)$. Let $\varepsilon > 0$, $t \in \N$. Then, for $d \geq 2t+ d_0$
    \begin{equation}\label{eq:HausDistRate}
        d\geq \gamma(n,\vp) 6^{2.5 n\text{\L}} t^{6n \text{\L}} \binom{n+t}{t}^{2.5 n\text{\L}} \varepsilon^{-2.5 n\text{\L}} \quad \text{implies} \quad d_H(\cM ^{(1)}(\bK)^t, \cL _d^{(1)}(\vp)^t) \leq \varepsilon.
    \end{equation}
\end{thm}

Inverting the expression (\ref{eq:HausDistRate}) for $\varepsilon$ gives the following corollary.

\begin{cor}\label{cor:ConvRateHausDist}
    Let $n,m,p_1,\ldots,p_m,K$ be as Theorem \ref{thm:BaldiMoment}. Let $t\in \N$. Then
    \begin{equation*}
        d_H(\cM ^{(1)}(\bK)^t, \cL _d^{(1)}(\vp)^t) \leq 6 \gamma(n,\vp)^{\frac{1}{2.5n\text{L}}} t^\frac{12}{5} \binom{n+t}{t} d^{-\frac{1}{2.5n\text{L}}}.
    \end{equation*}
    In particular, for each $L\in \cL _d^{(1)}(\vp)^t$ there exists a measure $\mu \in \cM ^{(1)}(\bK)^t$ such that for all $q \in \R[\bX]_t$ it holds
    \begin{equation}\label{eq:PseudomomentConvq}
        \abs{\int\limits q \; \od \mu - L(q)} \leq 6 \gamma(n,\vp)^{\frac{1}{2.5n\text{L}}} t^\frac{12}{5} \binom{n+t}{t} d^{-\frac{1}{2.5n\text{L}}} \norms{q}{\mathrm{coeff}}.
    \end{equation}
\end{cor}

\begin{proof}
    To show the first claim, we only need to re-arrange (\ref{eq:HausDistRate}) for $d$. For fixed $d\in \N$ we want to find minimal $\varepsilon> 0$ that satisfies $d\geq \gamma(n,\vp) 6^{2.5 n\text{\L}} t^{6n \text{\L}} \binom{n+t}{t}^{2.5 n\text{\L}} \varepsilon^{-2.5 n\text{\L}}$. That is, we choose $\varepsilon$ as
    \begin{equation*}
        \varepsilon := \left(\gamma(n,\vp) 6^{2.5 n\text{\L}} t^{6n \text{\L}} \binom{n+t}{t}^{2.5 n\text{\L}} \frac{1}{d} \right)^{\frac{1}{2.5 n\text{\L}}} > 0.
    \end{equation*}
    This gives exactly the expression in the first claim. To show (\ref{eq:PseudomomentConvq}), let $L \in \cL _d^{(1)}(\vp)^t$ and $q\in \R[\bX]_t$. By Lemma \ref{lem:M1KtClosed}, the set $\cM ^{(1)}(\bK)^t$ is compact. Thus, by using the first part of the statement, there exists a measure $\mu \in \cM(\bK)$ with $\mu(\bK) = 1$ such that
    \begin{equation}\label{eq:MinimizingMu}
        \norms{L - L_\mu}{\mathrm{op}}\leq 6 \gamma(n,\vp)^{\frac{1}{2.5n\text{L}}} t^\frac{12}{5} \binom{n+t}{t} d^{-\frac{1}{2.5n\text{L}}}.
    \end{equation}
    By applying the definition of the operator norm $\norms{\cdot}{\mathrm{coeff}}$ from (\ref{eq:DualNorm}) we get
    \begin{eqnarray*}
        \left| \int\limits q \; \od \mu - L(q) \right| & = & \left| L_\mu(q) - L(q) \right| \leq \norms{L_\mu\big|_{\R[\bX]_t} - L}{\mathrm{op}} \norms{q}{\mathrm{coeff}}\\
        & \overset{\text{(\ref{eq:MinimizingMu})}}{\leq} & 6 \gamma(n,\vp)^{\frac{1}{2.5n\text{L}}} t^\frac{12}{5} \binom{n+t}{t} d^{-\frac{1}{2.5n\text{L}}} \norms{q}{\mathrm{coeff}}.
    \end{eqnarray*}
    This shows (\ref{eq:PseudomomentConvq}).
\end{proof}

\paragraph{Atomic representation of measures}
    In this paragraph, we recall a result by Richter~\cite{richter1957parameterfreie} and Tchakaloff \cite{tchakaloff1957formules} (an English version can be found in~\cite[Proposition 4]{di2018multidimensional}) on quadrature rules. It states that, given finitely many moments of a measure $\mu\in \cM(\bK)$, we can find an atomic measure on $\bK$ matching those moments. 
    
    \begin{prop}[\cite{richter1957parameterfreie,tchakaloff1957formules}]\label{prop:Richter}
        Let $K\subset \R^n$ be a bounded measurable set and $\mu \in \cM(\bK)$ with $\mu(\bK) = 1$. Let $E\subset\R[\bX]$ be a finite dimensional subspace with $l := \mathrm{dim}(E)$. There exist an atomic measure $\nu := \sum\limits_{i = 1}^l a_i \delta_{\vx_i} \in \cM(\bK)$ with $\vx_1,\ldots,\vx_l \in \bK$ and $a_1,\ldots,a_l \in [0,1]$ such that
        \begin{equation*}
            \int\limits g \; \od \mu = \int\limits g \; \od \nu\; \left( = \sum\limits_{i = 1}^l a_i g(\vx_i)\right) \quad \text{for all } g\in E.
        \end{equation*}
    \end{prop}

%maybe it works also more directly using the positivstellensatz for linear monomials.
\end{subequations}

\section{Flatness and finite convergence}\label{Appendix:Flatness}

\begin{subequations}
In this part of the appendix, we want to illustrate some intuition on the concept of flatness that was essential in Sections \ref{sec:ConvRatesHierarchy} and \ref{Sec:ExtractExactMinimizers}. We recall the notion of flatness from Definition \ref{def:flatness}.

\begin{definition}[{r-Flatness}]
    Let $r,d \in \N$ with $r \leq d$. A linear form $L:\R[\bX]_{2d}\rightarrow \R$ is $r$-flat if it holds
    \begin{equation*}
        \rank \ M^{L} = \rank \ M^{L_{-r}}
    \end{equation*}
    where $L_{-r}:= L\big|_{\R[\bX]_{2(d-r)}}:\R[\bX]_{2(d-r)}\rightarrow \R$ denotes the restriction of $L$ to $\R[\bX]_{2(d-r)}$ and $M^L$ respectively $M^{L_{-r}}$ are the moment-matrices of $L$ respectively $L_{-r}$ defined in (\ref{eq:MLBilinearForm}).
\end{definition}

To illustrate why flatness appears naturally, we consider first the example of a Dirac measure.

\begin{example}\label{ex:DiracMeasureMomentMatrix}
    Let $\mu = \delta_\vz \in \cM(\R^n)$ be the Dirac measure in a point $\vz\in \R^n$. Let $M^\mu = (M^\mu_{\balpha,\bbeta})_{\balpha,\bbeta \in \N^n}$ be its infinite moment matrix, analog to (\ref{eq:MLBilinearForm}), given by
    \begin{equation*}
        M^\mu_{\balpha,\bbeta} := \int\limits \bX^{\balpha + \bbeta} \; \od \mu = \vz^{\balpha + \bbeta} = \vz^{\balpha} \cdot \vz^{\bbeta} \quad \text{for } \balpha,\bbeta \in \N^n.
    \end{equation*}
    In other words, we have $M^\mu = \vv \cdot \vv^T$ for $\vv = (\vz^{\balpha})_{\balpha \in \N^n}$, i.e. $M^\mu$ is a (infinite) matrix of rank one. In particular, for $d\in \N$, its principal minors $(M^\mu_{\balpha,\bbeta})_{|\balpha|,|\bbeta|\leq d}$ are also of rank one. This means, that for $d\geq 2$ the linear form $L:\R[\bX]_{2d}\rightarrow \R$ given by $p\mapsto \int\limits p \; \od \mu$ satisfies the flatness criterion from Definition \ref{def:flatness}.
\end{example}

The following proposition extends the previous example to finitely-atomic measures and is a variant of~\cite[Theorem 5.1 (ii)]{laurent2009sums}.

\begin{prop}\label{prop:RankMomentMatrix}
    Let $\mu = \sum\limits_{i = 1}^l a_i \delta_{\vx_i}$ be a finitely-atomic measure on $\R^n$ with $l\in \N$, pairwise distinct points $\vx_1,\ldots,\vx_l\in \R^n$ and weights $a_1,\ldots,a_l > 0$. Let $M^\mu = (M^\mu_{\balpha,\bbeta})_{\balpha,\bbeta \in \N^n}$ be the infinite moment matrix given by
    \begin{equation*}
        M^\mu_{\balpha,\bbeta} := \int\limits \bX^{\balpha + \bbeta} \; \od \mu \quad \text{for } \balpha,\bbeta \in \N^n.
    \end{equation*}
    Then, there exists $d_0\in \N$ with $d_0\geq 2(l-1)$ such that for all $d\geq d_0$ the principal minors $(M^\mu_{\balpha,\bbeta})_{|\balpha|,|\bbeta|\leq d}$ have rank $l$.
\end{prop}

\begin{proof}
    By linearity (in $\mu$) of the integral we have
    \begin{equation*}
        M^\mu = \sum\limits_{i = 1}^l a_i M^{\delta_{\vx_i}} = \sum\limits_{i = 1}^l a_i \vv_i \vv_i^T \quad \text{ with } \quad \vv_i := \left(\vx_i^{\balpha}\right)_{\balpha \in \N^n} \text{ for } i = 1,\ldots,l,
    \end{equation*}
    where we used the representation $M^{\delta_{\vx_i}} = \vv_i \vv_i^T$ from Example \ref{ex:DiracMeasureMomentMatrix}. The same holds for the principal minors; for $d\in \N$ we have
    \begin{equation*}
        (M^\mu_{\balpha,\bbeta})_{|\balpha|,|\bbeta|\leq d} = \sum\limits_{i = 1}^l a_i \vv^{(d)}_i (\vv^{(d)}_i)^T \quad \text{ for } \quad \vv^{(d)}_i := \left(\vx_i^{\balpha}\right)_{\abs{\balpha}\leq d}.
    \end{equation*}
    This shows that all principal minors $(M^\mu_{\balpha,\bbeta})_{|\balpha|,|\bbeta|\leq d}$ have at most rank $l$. Next, we show that for $2(l-1)\leq d\in \N$ they also have at last rank $l$. To do so, it suffices to check that the truncated moment-vectors $\vv^{(d)}_1,\ldots,\vv^{(d)}_l$ are linearly independent. Let $d\in \N$ with $d\geq 2(l-1)$ and assume that there exist $c_1,\ldots,c_l \in \R$ such that
    \begin{equation}\label{eq:DeltaLinearlyIndep}
        \vv:= \sum\limits_{i = 1}^l c_i \vv^{(d)}_i  = 0.
    \end{equation}
     We want to show that this implies $c_1 = \ldots = c_l = 0$. First, we note for $g = \sum\limits_{\abs{\balpha}\leq d} w_{\balpha} \bX^{\balpha} \in \R[\bX]_d$ with $\vw := \left( w_{\balpha}\right)_{\abs{\balpha}\leq d}$ we get\vspace{-2mm}
    \begin{equation}\label{eq:Interpolg}
        0 \overset{\text{(\ref{eq:DeltaLinearlyIndep})}}{=} \vw^T \vv = \sum\limits_{i = 1}^l c_i \vw^T \vv^{(d)}_i = \sum\limits_{i = 1}^l c_i \sum\limits_{\abs{\balpha}\leq d} w_{\balpha} \vx_i^{\balpha} = \sum\limits_{i = 1}^l c_i g(\vx_i).
    \end{equation}
    Consider the polynomials $g_j(\vx) := \prod\limits_{i \neq j} \|\vx-\vx_i\|^2 \in \R[\bX]_{2(l-1)} \subset \R[\bX]_{d}$ for $1\leq j \leq l$. It holds $g_j(\vx_i) = 0$ for all $j \neq i$ and $g_j(\vx_j) \neq 0$. We get from  (\ref{eq:Interpolg}) that $0 \overset{\text{(\ref{eq:Interpolg})}}{=} \sum\limits_{i = 1}^l c_i g_j(\vx_i) = c_jg_j(\vx_j)$ and hence $c_j = 0$. Since $1\leq j \leq l$ was arbitrary, we conclude the statement. \qedhere
\end{proof}

The truncation degree $d_0$ from Proposition \ref{prop:RankMomentMatrix} can be non-trivial. That is, taking $d_0$ to be the smallest $d$, for which $(M^\mu_{\alpha,\beta})_{\abs{\balpha},|\bbeta|\leq d}$ has at least $l$ rows respectively columns, might not be sufficient. We state a simple case in the following example.

\begin{example}
    Let $n = 2$ and consider the finitely-atomic measure $\mu = \frac{1}{3} \delta_{\vx_1} + \frac{1}{3} \delta_{\vx_2} + \frac{1}{3} \delta_{\vx_3}$ with atoms $\vx_1 = (-1,0)$, $\vx_2 = (0,0)$ and $\vx_3 = (1,0)$. For $d = 1$ we have
    %we have
    % \begin{equation*}
    %     \int\limits 1 \; \od \mu = 1, \quad \int\limits x_1 \; \od \mu = 0, \quad \int\limits x_2 \; \od \mu = 0, \quad \int\limits x_1^2 \; \od \mu = \frac{2}{3}, \quad \int\limits x_1 x_2 \; \od \mu = 0, \quad \int\limits x_2^2 \; \od \mu = 0.
    % \end{equation*}
    % Hence moment-matrix $A:=(M^\mu_{\alpha,\beta})_{\abs{\balpha},\abs{\bbeta}\leq d}$ reads
    \begin{equation*}
        A:= (M^\mu_{\alpha,\beta})_{\abs{\balpha},\abs{\bbeta}\leq d} = \left( \begin{array}{cccc}
             1 & 0 & 0 \\
             0 & \frac{2}{3} & 0\\
             0 & 0 & 0
        \end{array} \right)
    \end{equation*}
    and it holds $\rank (A) = 2 < 3 = l$. Thus, using only degree one polynomials, we do not recognize that $\mu$ has \textit{three} atoms. The reason is that the atoms $\vx_1,\vx_2,\vx_3$ lie on the zero-locus of the degree-one polynomial $p(y,z) := z$.
\end{example}

A more interesting result is the converse of Proposition \ref{prop:RankMomentMatrix}, which we state in the following theorem.

\begin{thm}[{\cite[Theorem 6.19]{lasserre2015introduction}}]\label{thm:flatness2}
    Let $L:\R[\bX]_{2d} \rightarrow \R$ be linear and the corresponding matrix $M^L$ from (\ref{eq:MLBilinearForm}) be positive semidefinite. If $L$ is $1$-flat according to Definition \ref{def:flatness}, then there exists an atomic measure $\mu \in \cM(\R^n)$ with at most $\rank (M^{L})$-many atoms with
    \begin{equation*}
        L(p) = \int\limits p \; \od \mu \quad \text{for all } p \in \R[\bX]_{2d}.
    \end{equation*}
\end{thm}

The result in Theorem \ref{thm:KrepMeasure} refines the Theorem \ref{thm:flatness2} to limit the support of $\mu$ to semialgebraic sets.

% \paragraph{Extracting atoms}
% Let $d\in \N$ with $d\geq \max\{\deg (f), \deg (p_1),\ldots,\deg (p_m)\}\leq d\in \N$ and $L:\R[\bX]_{2d}\rightarrow \R$ be a linear form that belongs to $\cQ_d(\vp)$ and is $r$-flat for $r = \max\limits_{i} \deg (p_i)$. By Theorem \ref{thm:KrepMeasure}, there exists a representing atomic measure $\mu = \sum\limits_{i}^l a_i \delta_{\vx_i} \in \cM(\bK)$ for $l = \rank M^L$ atomes $\vx_1,\ldots,\vx_l \in\bK$.

% Then for any $p \in \R[\bX]_{d}$ with $p(\vx_i) = 0$ for all $i = 1,\ldots,l$ it holds $B^L(q,p) = L(q \cdot p) = \int\limits q \cdot p \; \od \mu = \sum\limits_{i = 1}^l a_i q(\vx_i) \cdot p(\vx_i) = 0$. In other words, it holds
% \begin{equation*}
%     \vx_1,\ldots,\vx_l \in \{p^{-1}(\{
% \end{equation*}

% Let us consider $\cI := \{p\in \R[\bX]_d: L(p\cdot q) = 0 \text{for all } q \in \R[\bX]_d\}$ and the quotient ring $\cR:= \nicefrac{\R[\bX]_d}{\cI}$. On $\cR$ we consider the operator
\end{subequations}

\section{Simple moment bounds}

In this section, we provide two auxiliary results on computing and bounding moments, which have been used as technical ingredients in some parts of the paper.

\begin{lem}\label{lem:LipschitzMonomial}
    Let $n\in \N$, $x,y \in [0,1]^n$ and $\alpha \in \N^n$. It holds
    \begin{equation*}
        \abs{\vx^{\balpha} - \vy^{\balpha}} \leq |\balpha| \norms{\vx-\vy}{2}.
    \end{equation*}
\end{lem}

\begin{proof}
    We will show that the monomial $g(\vz) := \vz^{\balpha}$ has Lipschitz constant $ \abs{\balpha}$ on the hypercube $[0,1]^n$. The gradient $\nabla g = (\frac{\partial}{\partial_1} g,\ldots,\frac{\partial}{\partial_n} g)$ of $g$ is given by
    \begin{equation*}
        \frac{\partial}{\partial_i} g (\vz) = \alpha_i \vz^{\alpha_i - 1} \prod\limits_{j \neq i} \vz^{\alpha_j}, \quad \text{for } i = 1,\ldots,n.
    \end{equation*}
    For $\vz\in [0,1]^n$, we can bound $\norms{\nabla g(\vz)}{2}$ as
    \begin{eqnarray*}
        \norms{\nabla g(\vz)}{2}^2 & = & \sum\limits_{i = 1}^n |\alpha_i|^2 \underbrace{\abs{\vz^{\alpha_i - 1} \prod\limits_{j \neq i} \vz^{\alpha_j}}^2}_{\leq 1} \leq \sum\limits_{i = 1}^n |\alpha_i|^2 \leq \sum\limits_{i = 1}^n |\alpha_i| \abs{\balpha} = \abs{\balpha}^2.
    \end{eqnarray*}
    By the mean value theorem, we have
    \begin{equation*}
        \abs{\vx^{\balpha} - \vy^{\balpha}} = |g(\vx) - g(\vy)| \leq \norms{\vx-\vy}{2} \sup\limits_{\vz\in [-1,1]^n} \norms{\nabla g(\vz)}{2} \leq \norms{\vx-\vy}{2} \, |\balpha|.
    \end{equation*}
\end{proof}

The following lemma concerns vector valued integrals and can be found in any textbook on vector integrals, see for in instance~\cite[Section II.2 Theorem 4]{diestel1977vector}. For better self-containedness of this text, we state the result and its short proof.

\begin{lem}\label{lem:VectorIntNormBound}
    Let $h = (h_1,\ldots,h_n):\R^n\rightarrow \R^n$ be continuous, $\bK \subset \R^n$ be compact and $\mu \in \cM(\bK)$. It holds
    \begin{equation*}
        \norms{\left(\int\limits h_i\; \od \mu\right)_{i = 1,\ldots,n}}{2} \leq \int\limits \norms{h(\vx)}{2}\; \od \mu(\vx).
    \end{equation*}
\end{lem}

\begin{proof}
    Let $\langle \cdot,\cdot\rangle$ denote the Euclidean inner product. We have by linearity of the Euclidean inner product and by continuity of the integral
    \begin{eqnarray*}
        \norms{\left(\int\limits h_i\; \od \mu\right)_{i = 1,\ldots,n}}{2} & = & \sup\limits_{\substack{z \in \R^n\\ \norms{z}{2}\leq 1}} \left\langle z, \left(\int\limits h_i\; \od \mu\right)_{i = 1,\ldots,n} \right\rangle = \sup\limits_{\substack{z \in \R^n\\ \norms{z}{2}\leq 1}}  \int\limits \langle z,h(\vx)\rangle\; \od \mu(\vx)\\
        & \leq & \int\limits \sup\limits_{\substack{z \in \R^n\\ \norms{z}{2}\leq 1}} \langle z,h(\vx) \rangle\; \od \mu(\vx) = \int\limits \norms{h(\vx)}{2}\; \od \mu(\vx). 
    \end{eqnarray*}
\end{proof}

\section{Semialgebraic functions}
    In this text, we frequently make use of the \L{}ojasiewicz inequality, Theorem \ref{thm:loja}, which applies to semialgebraic functions. Here, we provide a minimal background on semialgebraic functions. Details can be found in any book on real algebra, such as~\cite{bochnak2013real} among others. We begin with semialgebraic sets.

    \begin{definition}[Semialgebraic sets]
        A set $\bK \subset \R^n$ is called semialgebraic if it is the finite union of sets of the form $\{\vx \in \R^n: p_1(\vx),\ldots,p_m(\vx) \geq 0\}$, $\{\vx \in \R^n: q_1(\vx),\ldots,q_k(\vx) > 0\}$ or $\{\vx \in \R^n: h_1(\vx),\ldots,h_l(\vx) = 0\}$, where $k,l,m\in \N$ and $p_1,\ldots,p_m,q_1,\ldots,q_k,h_1,\ldots,h_l \in \R[\bX]$.
    \end{definition}

    \begin{rem}
        A closed basic semialgebraic set $\bK = \cK(\vp) := \{\vx \in \R^n: p_1(\vx),\ldots,p_m(\vx) \geq 0\}$ from Definition \ref{def:BasicSemialgebraic} is semialgebraic.%Further, semialgebraic sets are closed under taking unions and intersections.
    \end{rem}

    The notion of a semialgebraic function is based on semialgebraic sets. 

    \begin{definition}[Semialgebraic function]\label{def:SemialgebraicFunction}
        Let $K\subset \R^n$ be a semialgebraic set. A function $f:\bK \rightarrow \R$ is called semialgebraic if its graph $\mathrm{Gr}(f) := \{(x,y) :\vx \in \bK, y = f(\vx)\} \subset \R^{n+1}$ is semialgebraic.
    \end{definition}
    
    Any polynomial function $f\in \R[\bX]$ is semialgebraic. But also non-smooth functions can be semialgebraic, such as the indicator function of a semialgebraic set $\bK$, i.e. $\mathds{1}_K:\R^n \rightarrow \R$ defined by $\mathds{1}_{\bK}(\vx) := 1$ for $\vx \in \bK$ and $\mathds{1}_{\bK}(\vx) := 0$ else. To see this note, that the graph of $\mathds{1}_{\bK}$ is given by
    \begin{equation*}
        \mathrm{Gr}(\mathds{1}_{\bK}) := \{(x,y) \in \R^{n+1}:\vx \in \bK, y - 1= 0\} \cup \{(x,y) \in \R^{n+1}: x\notin \bK, y = 0\}.
    \end{equation*}

    \begin{lem}\label{lem:DistSemialgebraic}
        Let $S\subset \R^n$ be a compact semialgebraic set. Then $\dist(\cdot,S):\R^n \rightarrow \R$ with $\dist(x,S):= \min\limits_{y\in S} \norms{x-y}{2}$ is a semialgebraic function.
    \end{lem}
    \begin{proof}
        The graph of $\dist$ is given by
        \begin{eqnarray*}
            \mathrm{Gr}(\dist) & = & \{(x,r) \in \R^{n+1} : \exists y \in S \; \norms{x-y}{2}^2 - r^2 = 0\}\cap \\
            & & \{(x,r) \in \R^{n+1} : \forall y \in S \; \norms{x-y}{2}^2 - r^2 \geq 0\}.
        \end{eqnarray*}
        By the quantifier elimination theorem~\cite[Proposition 5.2.2]{bochnak2013real} both of the sets on the right-hand side in the above equation are semialgebraic. Hence, also their intersection is semialgebraic. This proves the claim.
    \end{proof}
\bibliography{references}
\bibliographystyle{plain}

\end{document}